\documentclass  {article}
\usepackage{fullpage}
\usepackage{amsfonts}
\usepackage{amssymb}
\usepackage{amsmath}
\usepackage{amsthm}
\usepackage{hyperref}
\usepackage[nottoc,numbib]{tocbibind}
\usepackage  {graphicx}
\usepackage{caption}
\usepackage{subcaption}
\usepackage{algorithmic}
\usepackage{gensymb}
\usepackage[mathscr]{euscript}

\theoremstyle{plain}
\newtheorem{thm}{Theorem}[section]
\newtheorem{lem}[thm]{Lemma}

\theoremstyle{definition}

\theoremstyle{remark}

\DeclareGraphicsExtensions{.pdf,.png,.jpg}

\begin{document}
\title{ Inference of a Dyadic Measure and its Simplicial Geometry from Binary Feature Data  and  Application to Data Quality}
\thanks{The author gratefully acknowledges the CCICADA Center at Rutgers and the CCICADA Data Quality Team for providing the raw data quality statistics and thanks Christie Nelson for explaining them. The author also gratefully acknowledges  use of the open source Computational Homology Project software (CHomP)\cite{HaMi2014}\cite{chomp} and  thanks Shaun Harker for assistance with
 the installation and use of the software.This work was partially enabled by DIMACS through support from the National Science Foundation under grant number CCF-1445755.} 

\author{
Linda Ness\\
Rutgers University\\
\href{mailto:linda.ness@rutgers.edu}{linda.ness@rutgers.edu}
}

\date{April 29, 2017}
\maketitle

\newpage

\tableofcontents

\newpage
\section{Introduction }

Representation is a fundamental problem in data analysis. Data is pre-processed into one or more representations, preferably by an automated inference algorithm. The inferred representations determine the type of analysis, the results of the analysis and the privacy and security of the results of the data analysis. Exploiting theoretically-based representations improves reproducibility and replicability if the theoretical basis is algorithmizable as it enables re-use of algorithms, methods, and concepts and therefor clear documentation of the reasoning used in the analysis.   It can also clarify the privacy and security of the results of the data analysis. 

In this paper we will describe a new method for representing a data set for which  binary feature have been defined. The method algorithmically infers two related representations: a representation of the data set as a non-parametric multi-scale dyadic measure using the dyadic product formula representation\cite{BaNeJoSh16} \cite{FeKePi91} \cite{Ness16} and a summary of the simplicial geometry of the support of the measure in terms of the betti numbers of a variant of a nerve simplicial complex determined by the dyadic measure. Both of these representations are based on mathematical theory.  Typically data sets with binary features are analyzed using just the vector of values of the binary features (e.g. numerical linear algebra methods, machine learning classification methods, statistical methods). Our proposed representation method infers many more features from the raw binary features. This larger feature set more precisely describes the statistics and geometry of the data. This larger set of inferred canonical features can be used to statistically fuse sets of data samples with binary features for purposes of statistical analysis, prediction and decision making (including machine learning). For example, they can also be used to do multiscale hypothesis testing.  We will illustrate the method on a data quality data set, whose binary features are quality constraints for a raw data set. 

The analysis of the real data set led us to a Simplicial Binary Feature Representation Theorem for Dyadic Measures. The theoretical result is that the two concepts of counting measures on  binary feature sets  and dyadic sets with dyadic measures are essentially equivalent concepts for finite dyadic sets. The theorem also shows that there are canonical simplicial complexes determined by the measure are more abstract representations, independent of the order of the features,  whose betti numbers statistics have a differential privacy property. We also compare the work here to some related work on measures on tree structured sets, recent multi-resolution theory, and computational topology and suggest some research directions. As part of this discussion, we prove a Representation Lemma for Measures on Tree Structured Sets, which removes the dyadic restriction and makes explicit  some of the resulting complexities. Data sufficient to reproduce the analysis of the example is included in Appendix B.

\subsection {Representing Data Sets with Binary Features}

Data sets are often preprocessed into a set of binary features. Each data point is then represented by (mapped to) the set of 1's and 0's  indicating the value of each binary feature on the data point. Dually, each binary feature determines two subsets of data points, which we will call feature sets: the subsets of data points $F_i = 0$ and $F_i = 1$ on which the $i^{th}$ feature $F_i$ has value 0 and value 1. Typically there is not a 1-1 mapping between the data points and the feature sets; in some cases there are many fewer features than there are data points, as is the case in the data quality data set. The counting measure of the data set is the  number of data points mapped to a possible value of the set of features. It provides more information about the data set than just the image of feature mapping if the feature mapping is not one to one. The counting measure is a measure on the sigma-algebra generated by the feature sets. Although this sigma algebra contains $2^n$ sets, one for each possible value of the set of $n$ features, only relatively few of these sets support the measure (i.e.  have non-zero measure) for many real world data sets, so it is often practical to compute this measure.   When an order is chosen on the binary features, the binary feature sets form an ordered binary tree. The counting measure then has a unique dyadic product formula representation with one product coefficient parameter for each non-leaf node in the tree.  The product coefficient parameters are unique, given the order of  the binary features,  and provide a  unique set of multi-scale statistics which can be used to represent and visualize the data set. The geometry of the the hyperplanes determined by the binary features can be complex. One way to characterize this geometry is topological using simplicial complexes and their betti numbers. (A simplicial complex is a collection of sets, all of whose subsets are contained in the collection. The support of the measure is a set on which it is positive and not zero.)  We prove that all  dyadic sets can be pre-processed into a set of binary features and all dyadic measures determine a binary features which then determine canonical nerve complexes, so canonical betti number statistics can be defined and computed for all dyadic measures. 

\subsection {Data Quality Application Overview}

We computed this representation (product coefficients  and betti numbers)  for 5 data sets whose binary features were a common set of data quality constraints. The binary feature representation of each data item represented its data quality. If the value of a feature was $0$ on a data item, the data item violated the quality constraint. The CCICADA Data Quality team had previously specified 30 complex domain specific data quality constraints for the set of data after conducting extensive interviews with domain experts and analyzing the raw data set . Each data item was a row in a relational data base.  The original data consisted of  5 samples of data from 5  data sources. The first source was the composite source; the other four sources were distinct. The composite source consisted of $~ 4000$ data items.  The features were ordered by the counts in the composite source. Only the counts and the  binary feature values were made available to us for this analysis. The raw data and its description remained private. Our goal was to find representations that would enable us to describe the mutual constraint violation sets and understand their relationships concisely and canonically. In other words, we wanted to describe and understand the intersections of the 30 binary feature sets. While there were  potentially $2^{30}$ possible types of mutual constraint violation sets, only 219 maximal constraint violation sets were non-empty in the composite source. These 219 sets are  the canonical discrete disjoint events for the data set. Note that the mutual constraint violation sets can also be viewed as edges of a hypergraph whose vertices are the individual data quality items. 

For each these 5 data sources, the source data was represented as a  non-parametric multi-scale dyadic measure and its product coefficient parameters was computed. Also, for each data source the associated nerve simplicial complex and its betti numbers  were computed. Both of these representations separately distinguished the data sources in this case. The 5 product coefficient parameters for the five measures were visualized using 2-dimensional daywheels. None of the five sources had a data item violating more than ten constraints, there were no simplicial complexes of more than $10$ dimensions, so only 10 betti numbers needed to be computed. The betti numbers were computed using the open source Computational Homology Project software (CHomP)\cite{HaMi2014}\cite{chomp}.

\section { The Binary Tree Representation for Dyadic Measures} \label{dyadicmeasures}
\subsection {Ordered Binary Feature Sets Determine Dyadic Sets and Vice Versa}  \label{dyadicsets}
An ordered set $\mathcal{F} = \{F_1, ...., F_{maxscale}\} $ of binary features defined on a set $D$ determines an  ordered binary tree $\mathcal {T} $  with levels $0$ through $maxscale$ and associates a subset $S(n)$  of $D$ defined by feature equations with each node $n$ in the tree. The maximum scale $maxscale$  can be finite or infinite. The root node set corresponds to the whole set $D$. The left and right child sets at distance 1 (scale 1 or level 1) from the root correspond to the first feature set $F_1 = 0$ and its complement $F_i = 1$. For a node $n$ at distance $i$ from the root node, reached from the root by the binary path $ L(n) = p_1 .....p_i$  the assicated node set $S(n)$ is defined by the $i$ feature feature equations: 
\begin{equation} \label{nodeequation}
S(n)   = \cap_{j = 1, ... i, p_j = L(n)_j}  (F_j = p_j)
\end{equation} 

Thus the node set is the intersection of the first $i$ feature sets $F_j = p_j$ using the values $p_j$  specified in the path $L(n)$ to the node. 
There are  $2^i$  nodes $n$ at level $i$ (i.e. distance $i$ from the root)  corresponding to the $2^i$ sets obtained by intersection the first $i$ feature sets $F_i = 0$  and their complements $F_i = 1$ .  The collection of node sets $\mathcal{S} = \{S(n)\}$ is therefore a \textit{dyadic set}. The definition of a dyadic set $X$ is: an ordered binary set system for $X$  consisting of disjoint left and right child subsets for each parent set, whose union is the parent set, and whose root set is $X$.  

The feature sets $F_i = 0/1$  can be obtained from the collection of node sets $\mathcal{S} = \{S(n)\}$ at level $i$  by union operations.  The feature set $F_i = 0$  is the union of the left child sets of the level $i-1$ node sets $S(n)$. The feature set  $F_i = 1$  is the union of the right child sets of the level $i-1$ node sets $S(n)$. In fact, these union relations can be used to define an ordered set $\mathcal{F} = \{F_1, ...., F_{maxscale}\} $ of binary features for each dyadic set.  
Thus for any dyadic set $(X,\mathcal{S})$  the  sigma algebra $ \Sigma(\mathcal{ S})  $  generated by the node sets equals $ \Sigma(\mathcal{ F) } $ the signma algebra generated by the feature sets. 
\begin{equation} 
\Sigma(\mathcal {S}) = \Sigma(\mathcal {F})
\end {equation}
Since a sigma algebra consists of sets generated by intersection, complementation, and countably infinite union from the generating set. Equation \ref{nodeequation} shows $ \Sigma(\mathcal {S}) \subseteq \Sigma(\mathcal {F})$. and the union equations defining the features show $ \Sigma(\mathcal {F}) \subseteq \Sigma(\mathcal {S})$

 \subsection{ Geometry of a Dyadic Set}
 
In addition to the ordered binary tree geometry of a dyadic set $D$, there are two additional geometric  views.
 First, there is the geometry of the image of the dyadic set under the mapping $D \rightarrow \mathit{Binary\; Feature\; Space}$, where $d \rightarrow  (F_1(d), ...., F_{maxscale}(d))$. Here $\mathit{Binary\; Feature\; Space}$ is a vector space of dimension $maxscale$ where the scalar field is $\mathbb{F}_2 = \{0,1\}$.  For a node at level $i$, the image of the node set $S(n) \subseteq D$ under the mapping is a linear space of codimension $i$ because it is defined by $i$ linear binary feature equations. The image of a leaf is the binary feature vector defined by the $maxscale$ equations defining the leaf node set. 
Next there is the geometry of the image of the dyadic set under the canonical mapping $D \rightarrow [0,1] $ which takes a node set $S(n)$ at level $i$ to a dyadic interval of length $2^{-i}$. Recall that unit interval $[0,1]$ is a dyadic set with the binary set system $ \mathcal {S}$  consisting of the  dyadic unit intervals
 $[i\cdot 2^{-j}, (i+1)\cdot 2^{-j})$ for $i = 0, ..., 2^{j} - 2$ and $[i\cdot 2^{-j}, (i+1)\cdot 2^{-j}]$ for $i = 2^{j} - 1, j = 0, 1, ....2^i - 1$. Under this mapping 
 the $i^{th}$ feature set $F_i =0$ is the set $\{x \in [0,1]: i^{th}$ binary  digit is 0 $\}$ and its complement $F_i =1$ is the set $\{x \in [0,1]: i^{th}$ binary digit is 1 $\} $. For a node $n$ at level $i$  whose path $P$  from the root has labels $P = (p_1, ..., p_i)$, the associated node set $S(n) \subseteq D $ 
 is the set $\{x \in [0,1]: j^{th}$ binary digit is $p_j , j = 1, ..., i\}$. For infinite trees, this shows that the image of a dyadic set under the canonical mapping $D \rightarrow [0,1] $ is a Cantor-like set. The images of the leaf sets are in the lexicographic order determined by the feature set indices and their binary values. 

\subsection {Product Coefficients for Measures on Dyadic Sets} \label{pcs}
Let $\mu$ denote a measure on a dyadic set  $(D, \mathcal {S})$ . Then $\mu$ assigns non-negative  values to the node sets in the binary tree in a manner which satisfies the additive property:  if $L(n)$ and $R(n)$ are the left and right child nodes of $n$, 
$\mu(L(n)) + \mu(R(n)) = \mu(n)$, i.e. the sum of the measures of the left and right child nodes sets is the measure of the parent node set. Counting measure on a finite sample of $D$ is one example of a dyadic measure. Counting measure is the number of sample data items in each node set $S(n)$.

 For each non-leaf node $n$ of the tree, 
 the \textit{product coefficient parameter} is defined as a solution to the following equations:
\begin {equation}
\mu(L(n)) = \frac{1}{2}(1 + a_n)\mu(n)
\end {equation}
\begin {equation}
\mu(R(n)) = \frac{1}{2}(1 - a_n)\mu(n)
\end{equation}
A unique solution to the equations exists if $\mu(n) \neq 0$. If $\mu(n) = 0$ the solution is not unique. To make the product coefficients unique we adopt the convention that whenever one of the 'halves' of a dyadic set has measure zero, the product coefficients for all of the descendant sets  of the zero measure 'half' have zero product coefficients. This convention implies that if $\mu(n) = 0$  the solution $a_S = 0$ is chosen. 
Thus the product coefficient parameter $a_n$ is defined by the simple formulas:
\begin{equation}
a_n =
\begin{cases}
0 & \text{ if } \mu(n) = 0 \\
\frac {\mu(L(n)) - \mu(R(n))} {\mu(n)} & \text{ if } \mu(n) \neq 0 \\
\end{cases}
\end{equation}

A factor can also be associated to each edge  of the tree. The factor is: 
\begin {itemize} 
\item $1 + a_{P(n)}$ for a left edge emanating from the parent node $P(n)$
\item $1 - a_{P(n)}$ for a right edge emanating from the parent node $P(n)$
\end {itemize}

The product coefficient parameters uniquely determine the measure $\mu$  by the Dyadic Product Formula Representation ( Lemma 2.1)\cite{BaNeJoSh16}, even when the binary tree is infinite. The basic observation is that $\mu(S(n))$ equals $\mu(D)$ multipled by the product of the factors from the root to a node $n$  divided by $2^{-scale(n)}$. This evaluation process can be summarized by a product formula involving Haar-like functions. For each node set $S(n)$ define a Haar-like function $h_{S(n)}$ by

\begin{equation}
\begin{cases}
h_{S(n)}  = 1\ on\ S(L(n))\\
h_{S(n)}  = -1\ on\ S(R(n)\\
h_{S(n)} = 0\ elsewhere\\
\end{cases}
\end{equation}

The Dyadic Product Formula Representation is:
\begin{equation} \label{dyadicproductformula}
\mu = \mu(D) \cdot \prod_{n \in \mathcal S} (1 + a_n\cdot h_{S(n)}) \cdot dy
\end {equation}
where $dy$ is the dyadic measure which assigns a measure of $2^{-i}$ to node sets at scale $i$. It holds even when the binary tree is infinite. It says that value of the measure for the each node set $S_n$ 
can be computed in terms of the product coefficient parameters of its ancestor nodes by multiplying the factors associated with the node set $S(n)$ and its ancestor nodes.

Since the product coefficient parameters associated with the nodes  are all in the interval $[-1,1]$ they can be color coded and  visualized by day wheel figures. The day wheel figures for five data quality sources are shown in Figures \ref{fig:source1} and \ref{fig:source2to5} 
Another part of the Dyadic Product Formula Representation Lemma  is: any assignment of product coefficient parameters from the interval $[-1,1]$ following the convention determines a unique dyadic measure on $D$.

\subsection {The Multiscale Support of a Dyadic Measure} \label{support}
A support set for a measure  $\mu$ is any measurable set $S$ such that the measure of its complement $S'$ is zero, i.e.  $\mu(S') = 0$. For a dyadic measure $\mu$ on a dyadic set $(D, \mathcal{S})$ of finite scale $maxscale$, there is a unique smallest measurable support set. It is the union of the leaf sets which have non-zero measure. From the point of view of binary features, these are the leaf sets defined by a binary root to leaf path $L$ whose $maxscale $ equations $F_i = p_i(L), i  = 1, ..., maxscale + 1$ have a solution set with positive measure (strictly greater than 0). They define the discrete event set of occurrences for this measure.
For each scale $i >= 0$ if a dyadic set, the measure on the sub set system of $\mathcal{S}$ consisting of sets of scale $i$ or less therefore also has a unique support set -- the leaves of the tree of height $i$ which have positive measure. These define the multiscale support sets for the measure. 

The support of the measure has several geometric interpretations, in addition to the intrinsic tree geometry. The image of the leaf sets in the support with positive measure under the mapping $D \rightarrow \mathit{Binary\; Feature\; Space}$ is a subset of the binary feature vectors (i.e. a set of points in binary feature space). The image of the leaf sets in the support under the mapping $D \rightarrow [0,1] $ is a subset of dyadic intervals of length $2^{-maxscale - 1}$. 
 
\subsection{Order-Dependence, Bayes Formulas, and Invariant Measures} 

An important point is that the set of product coefficients depends on the order of the binary features. Every finite permutation $g$ of the features determines another set of product coefficient representation of the measure and hence another set of product coefficients. The two sets of product coefficients are related by a Bayes formula, so the new parameters are related by rational algebraic formulas to the original set of parameters. One quantitative formulation of this Bayes-type rule is given in Appendix 2 of version 2 of \cite{BaNeJoSh16}.

For example, if there are $n$ binary features, the symmetric group  $\mathcal{S}_m$ permutes $m$ features with indices $i_1 < ... < i_m <= n$  it permutes the binary root to leaf paths $P = (p_1, ....., p_n)$  by permuting the path label elements with indices $i_1 < ... < i_m <= n$, leaving the rest fixed. Explicitly, for  a permutation  $g \in \mathcal{S}_m$ , $P = (p_1, ....., p_n)$,  $g(P) = (q_1, ...., q_n)$ where 

\begin{equation}
\begin{cases}
q_i = p_i  \;  if \; i \not\in {i_1, ..., i_m} \\
q_{g(i)} = p_i   \; if  \; i \in {i_1, ..., i_m}\\
\end{cases}
\end{equation}

This implies that Bayes rule for this group action falls into 3 cases. The product coefficients determined by the new order of features are the same as for the old order of features for levels $ 0 <= i < i_1$. Since the subtrees rooted at level $i_m$ nodes are permuted, the product coefficients for levels $i > i_m$ are also permuted, because the are completely determined by their level $i_m$ ancestor. The  product coefficients for nodes at levels $i_1 .... i_m$ for the new order, can be recomputed top-down using the product formula for the old order  for the measures for the left and right child nodes.  

Each group action determines a measure invariant to the group action, which can be explicitly computing by averaging the measures at the leaf nodes in an orbit of the group action on the root to leaf paths. Additional research and experimentation will be  required to determine how to exploit these invariant measures. 

\section {Betti Numbers for Simplicial Complexes}

\subsection{Three Canonical Simplicial Complexes for a Dyadic Measure} \label{nervecomplexes}

An abstract simplicial complex $\mathcal{S}$ is a family of non-empty sets, all of whose non-empty subsets are in the family. For example, the collection of sets consisting of $\{1,2\}, \{3,4\}, \{2\}, \{3\}$,  and $\{4\}$ is not a simplicial complex because the set $\{1\}$ is not in the family.  A family of subsets $\mathcal{F} = \{S_1, ....., S_n\}$ of some universe $U$ determines a simplicial complex  called its $\it{nerve}$ $\it{complex}$. The nerve complex consists of the sets of indices of all subsets which have non-empty intersection.
For example if $\mathcal{F} = \{S_1, ....., S_3\}$   where $S_1 = \{a,b,c\}$, $S_2 = \{b,c,d\}$, $S_3 = \{a,e\}$, its nerve complex $\mathcal \{N\} = \{\{1\}, \{2\}, \{3\}, \{1,2\}, \{1,3\}\}$. 

A dyadic measure $\mu$ on a finite dyadic set $(D, \mathcal{S})$ determines 3 different canonical simplicial complexes. Let $F_1, ...., F_{maxscale}$ denote the binary features for $\mu$.  
Let  $\mathcal{N}(\mu)$ denote the collection of sets $P$ of ordered pairs $(i,b)$  such $\mu(\cap_{(i,b) \in P} (F_i = b) )  \neq 0 $. The ordered pairs are indices and are lexicographically ordered. 
Let  $\mathcal{N}_0(\mu)$ denote the collection of sets $I$ of indices $i$ such $\mu(\cap_{i \in I} (F_i = 0) )  \neq 0 $. 
Let  $\mathcal{N}_1(\mu)$ denote the collection of sets $I$ of indices $i$ such $\mu(\cap_{i \in I} (F_i = 1) )  \neq 0 $. 
The sets in each of these collections are ordered by containment. There are two observations. First $\mathcal{N}(\mu)$, $\mathcal{N}_0(\mu)$, and $\mathcal{N}_1(\mu)$ are each simplicial complexes, because if a set determined by a set of indices has positive measure, so does a subset determined by a subset of indices. They are analagous to nerve complexes, but the intersection condition of having a non-zero measure is more restrictive than having a non-empty intersection. Second, the maximal sets for each of these simplicial complexes (i.e. the maximal faces) correspond to leaf node sets in the support of the $\mu$. The maximal faces of each are sufficient to completely determine the simplicial complex. Each set in the multiscale support of $\mu$ is represented by a simplex in $\mathcal{N}(\mu)$.  Each leaf node set in the support of $\mu$ is represented by a maximal set in $\mathcal{N}(\mu)$.  Each leaf node set in the support of $\mu$ is represented by a maximal simplex  in $\mathcal{N}_0(\mu)$, unless $\mu(\cap_{i \in [1,..., \mathit{maxscale}} (F_i = 1) )  \neq 0 $. In this case the maximal sets of  $\mathcal{N}_0(\mu)$ are missing this one leaf set in the support. 
Similarly, each leaf node set in the support of $\mu$ is represented by a maximal set in $\mathcal{N}_1(\mu)$, unless $\mu(\cap_{i \in [1,..., \mathit{maxscale}} (F_i = 0) )  \neq 0 $. In this case the maximal sets of  $\mathcal{N}_1(\mu)$ are missing this one leaf set in the support. In the "generic case" the maximal faces of these simplicial complexes correspond exactly to the support of the dyadic measure $\mu$. In the data quality example, we used $\mathcal{N}_0(\mu)$, because these sets correspond to data quality constraint violations and every data vector in the data set violated at least one data quality constraint. 
 
\subsection{The Betti Numbers of Simplicial Complex} \label{bettinumbers}
We will next quickly recall the definition of betti numbers of a simplicial complex so that we can use them to describe the simplicial geometry of the nerve simplicial complexes. The definition of the betti numbers uses a few more terms and concepts \cite{KaMi}. Each set in a simplicial complex is often called a $\textit{face}$; the $\textit{dimension of  a face}$ is 1 less than the cardinality of set. Faces of dimension 0 are called $\it{vertices}$.  The $\textit{dimension of the simplicial complex}$ is the maximum of the dimension of the faces. We will only work with finite simplicial complexes.  
A simplicial complex $\mathcal{S}$ is abstractly summarized by its $\it{betti\ numbers}$ $\beta_i, i = 0, ... dim(\mathcal{S})$. These are computed algebraically as the rank of the quotient of two free abelian groups: the cycle group and the boundary group for each dimension.  The groups are defined in terms of a boundary map which maps a face to a formal sum of  its faces (of dimension one less) . More generally the boundary map maps a formal sum of faces (an element of the free abeilan group called $\it{chains}$ ) to the formal sum of the boundaries of its faces. The subtlety is orientation and signs in the boundary map. We can finesse the first,  if we assume the vertices are ordered and the face subsets are listed using this ordering. Then the boundary map maps a face to a signed formal sum of its faces. The subface of each face omits just one vertex of the face. If this is the $i^{th}$ largest vertex of the face, the sign $(-1)^{i}$ can be used. A fundamental lemma in simplicial homology (and an easy calculation) shows the composition of the two successive boundary maps determined by this choice is the zero map. 
For each dimension $i, 0 <= i <= dim(\mathcal{S})$ the cycle group $Z_i$  is the kernel of the boundary map on the free group of chains of that dimension, and the boundary group $B_i$ is the image of the boundary map from the free group of chains of one higher dimension. The i{th} betti number is the rank of the quotient group $Z_i/B_i$. 
\begin{equation}
\beta_i(\mathcal{S}) = rank(\frac{Z_i}{B_i})
\end{equation}
In many cases (i.e. if there is no torsion) this rank is just the difference of the ranks, $rank(Z_i) - rank(B_i)$. The rank of a free group is the number of basis elements, analogous to dimension for vector spaces.  
The betti numbers do not provide any information about the existence or non-existence of torsion elements (elements for which a multiple is 0) in the quotient groups $\frac{Z_i}{B_i}$. In this paper we used integer coefficients in the free group calculations.
An important fact is: the betti numbers do not depend on the order of the features. They are invariants of the action of permutation groups which reorder the feature set. 
\subsection{Simplicial Geometric Explanation of Betti Numbers}
The betti numbers summarize the simplicial geometry of the simplicial complex. Think of $0$-simplices as points, $1$-
simplices as edges, $2$-simplices as triangles, $3$-simplices as tetrahedrons etc. This geometry is relational or combinatorial, generalizing graph geometry if the simplicial complex is the nerve complex determined by an ordered set of binary features. It is not metric or manifold geometry, although for other choices of simplicial complexes (e.g. the RIPS complex)  on data sets which are subsets of a metric space, betti numbers can provide geometric information about metric and manifold geometry 
\cite{BoMuTu}.  

The $0^{th}$ betti number has a very intuitive geometric interpretation. The $0$-dimensional faces are the vertices of the simplicial complex. For a nerve complex determined by  binary feature sets $F_i = 0$ and $F_i = 1$, they are just the indices of  these  sets: (i,0) and (i,1). The image of the boundary map consists of linear combinations  of indices  associated with pairs of feature sets which have a non-empty intersection: $(F_i = p_i )\cap (F_k = p_k )\neq \emptyset$, where $i > k$. The image of the boundary map associated with this pair of features sets is: $(k,p_k) - (i,p_i)$. Elements of the $0$-dimensional  quotient group $\frac{Z_0}{B_0}$ are equivalence classes. The representatives $(k,p_k) + B_0$ and $(i,p_i) + B_0$ are in the same equivalence class if they are associated with a pair of feature sets which has a non-empty intersection. Thus the $0^{th}$ betti number  $\beta (0)$ is the number of connected components, where the connectivity relation is determined by non-empty intersection of pairs of sets. 

The highest dimensional betti number is $\beta(d)$, where $d$ is the maximum of the dimension of the faces. It is also quite simple to explain. Since there are no boundaries $\beta(d)$ is  the  rank of $Z(d)$ the free group of cycles of dimension $d$. 
Cycles occur only when there are constraints among the boundaries of the $d$-dimensional faces, i.e. constraints among the the alternating boundary linear combinations of  subfaces of faces. Cycles occur when there is a collection of $d$-dimensional simplices (each specified by $d+1$ indices) which are all of the faces of a $(d+1)$-dimensional simplex. Thus the constraints are complex. Two independent cycles occur when: there are two collections of  $d + 2$ $d$-dimensional simplices (each specified by $d+1$ indices); each collection consists of all of the faces of a $(d+1)$-dimensional simplex; the two collections have no vertices in common. Since the cycle constraints are complex, often there are none: e.g. if there are very few faces of maximum dimension. So the typical value for the highest betti number is 0 if $d$ is large. 

An intermediate betti number $\beta(i), 0 < i < d$ is the number of equivalence classes of cycles of dimension $i$, where two cycles (i.e. two formal linear combinations of $i$ dimension simplices whose boundaries equal $0$) are equivalent if there difference is a boundary of linear combination of simplices of dimension $i+1$. For example, suppose there are two independent cycles - e.g. there are two collections of  $i + 2$ $d$-dimensional faces (each specified by $i+1$ indices); each collection consists of all of the faces of an $(i+1)$-dimensional simplex; the two collections have no vertices in common;  there is one $i + 1$ dimensional simplex whose faces are the simplices in the first collection; there is no   $i + 1$ dimensional simplex whose faces are the simplices in the second collection. Then $\beta(i) = 1$ (and not $2$) because the first collection is in the image of the boundary of the higher $i+1$-dimensional simplex. 
The equivalence relation for each of these intermediate dimensions can be intuitively viewed as a type of higher-level connectivity determined by simplicial geometry.  Thus the intermediate betti numbers intuitively indicate the number of groups of faces of potential but not realized higher-dimensional simplices. 
\subsection{Example}
Table~\ref{tab:exsimp} lists the faces which generate a three dimensional simplicial complex. Each line of the table is the set of vertices for a face of the complex. The complete simplicial complex consists of all of the subsets of the listed faces. 
\begin{table}[ht]
\caption{Example - List of Faces Generating A  Simplicial Complex}
\centering
\begin{tabular}{r r r r}

1 & 2  &3 \\
1 & 3 & 4 \\
2 & 3  &4 \\
1 & 2 & 4\\
5 & 6 & 7 \\
5 & 6 & 8\\
6 & 7 & 8\\
5 & 7 & 8 \\
9 & 10 & 11\\
9 & 10 & 12\\
10  &11 & 12\\
9 & 11  &12\\
9 & 10  &11 & 12\\
13 &14\\
14 & 15\\
13 &15\\

\end{tabular}
\label{tab:exsimp}
\end{table}
The betti numbers for the simpicial complex are listed in Table~\ref{tab:exbettinumbers}. There are four groups of disjoint faces. The first group of two dimensional faces  involves vertices $1, 2, 3,4$. The second group of two dimensional faces involves vertices $5,6,7, 8$. The third group of two dimensional faces involves vertices $9, 10, 11, 12$. There is one 3 dimensional face: $\{9,10,11,12\}$ The fourth group consists of one dimensional faces and involves vertices $13, 14,15$. Each of the four groups determines a distinct connected component, so $\beta(0) = 4$. The third dimensional betti number $\beta(3) = 0$ because there is only one three-dimensional face and there is no linear constraint among its faces; hence there are no three dimensional cycles. The second betti number $\beta(2) = 2$ because, although there are three independent groups of two dimensional faces, each of which determines an independent cycle,  the group involving vertices $9, 10, 11,12$ is equivalent to the boundary of the three dimensional face. The first betti number $\beta(1) = 1$ because the fourth group of faces consists of all of the faces of a potential two dimensional simplex which does not occur; hence this group determines a two-dimensional cycle which is independent from two-dimensional cycles determined by the other groups of faces. Note that if the indices in the example corresponded to indices of sets determining a nerve simplicial complex, a potential two dimensional simplex would not occur if there were 3 sets which did not have a common intersection, but each pair of the three sets did have a common intersection. 
\begin{table}[ht]
\caption{Betti Numbers for the Simplicial Complex }
\centering
\begin{tabular}{c r r r r}
Dimensions  & 0 & 1 & 2 & 3\\[0.5ex] 
Betti Numbers & 4 & 1 & 2 & 0\\
\end{tabular}
\label{tab:exbettinumbers}
\end{table}

This example could arise as the nerve complex of a set of 15 binary features. For example, if we only consider the nerve complex determined by $F_i = 0$, as will be the case in the data quality example because there we only want to study the simplicial geometry of the constraint violation sets, a simplex ${i,j,k}$ corresponds to a binary feature vector with $0$'s in the positions $i,j,k$. In the data quality example, this would correspond to the set consisting of data elements which violated these three constraints. Thus the betti numbers describe the simplicial geometry of the types of binary feature vectors. The faces listed in the example would correspond to constraint violation sets for data item. A face would be listed if at least one data item violated the set of constraints indicated by the face vertices. 

\section {Simplicial Binary Feature Representation Theorem for Dyadic Measures}
In this section we recall and combine  the statements that have been proved in the previous sections. 
First recall a dyadic set $D$ is a set together with an ordered binary set system $\mathcal{S}$  consisting of disjoint left and right child subsets for each parent set, whose union is the parent set, and whose root set is $D$. Next recall a dyadic measure is a measure $\mu$ on $(D, \Sigma(\mathcal{S}))$ which takes non-negative values on all the sets in $\Sigma(\mathcal{S})$, the sigma algebra generated by sets in $\mathcal{S}$.
  
 \begin{lem} [Simplicial Binary Feature Representation Theorem for Dyadic Measures]
 
Let $(D,\mathcal S)$ be a dyadic set and let $\Sigma(\mathcal{S})$ be the sigma algebra of sets generated by sets in $\mathcal{S}$

\begin {enumerate}
	\item Let  $F_i:D \rightarrow \{0,1\}$  denote the binary valued feature function which assigns value 0 to elements of sets which are left children of the parent nodes at level $i-1$ and assigns value 1 to elements of sets which are right children of the parent nodes at 	level $i-1$. Then the collection of sets defined by equations
			\begin{equation}
			 \cap_{j = 1, ... i}  (F_j = p_j)  
			\end{equation} 
	form a dyadic set system which equals $\mathcal{S}$. Here $ i >0$ ranges over levels in the binary tree determined by $\mathcal{S}$  and $p_1, ...., p_i$ ranges over binary strings of length $i$.
	\item Let $\mathcal{F}$ denote the set of feature sets $F_i = 0/1$. Then the sigma algebras generated by $\mathcal{F}$ and $\mathcal{S}$ are equal, $\Sigma( \mathcal{S}) = \Sigma(\mathcal{F})$.	
	\item Assume the dyadic set system $\mathcal{S}$ is finite.   Let  $\mathcal{N}(\mu)$ denote the collection of sets $P$ of ordered pairs $(i,b)$  such that $\mu(\cap_{(i,b) \in P} (F_i = b) )  \neq 0 $. Let  $\mathcal{N}_0(\mu)$ denote the collection of sets $I$ of indices $i$ such that $\mu(\cap_{i \in I} (F_i = 0) )  \neq 	0 $. Let  $\mathcal{N}_1(\mu)$ denote the collection of sets $I$ of indices $i$ such $\mu(\cap_{i \in I} (F_i = 1) )  \neq 0 $.  Then $\mathcal{N}(\mu)$, $\mathcal{N}_0(\mu)$, and $\mathcal{N}_1(\mu)$  are simplicial complexes. The simplicial complexes are independent of the order of the binary features. The indices for each maximal face in the simplicial complexes determine a leaf set in the  support of $\mu$. The support of $\mu$ is the union of maximal faces of the simplicial complex $\mathcal{N}(\mu)$.  Each leaf node set in the support of $\mu$ is represented by a maximal simplex  in $\mathcal{N}_0(\mu)$, unless $\mu(\cap_{i \in [1,..., \mathit{maxscale}} (F_i = 0) )  \neq 0 $. Each leaf node set in the support of $\mu$ is represented by a maximal simplex  in $\mathcal{N}_1(\mu)$, unless $\mu(\cap_{i \in [1,..., \mathit{maxscale}} (F_i = 1) )  \neq 0 $.
 	\item \textit{Privacy Property:}  If $\mu_1$ and $\mu_2$  are dyadic measures on $(D,\mathcal S)$ have the same support (i.e. if they are positive on the same leaf sets at level $maxscale + 1$) , then determine the same simplicial complexes $\mathcal{N}(\mu)$, $\mathcal{N}_0(\mu)$, and $\mathcal{N}_1(\mu)$, and thus have the same betti numbers for each of these simplicial complexes. 
	\item If $E$ is a set with an ordered set of  binary features $F_i:E \rightarrow \{0,1\}$, i = 1, ... n , E is a dyadic set. Any finite subset of $E$ determines a counting measure on $E$. The counting measure is  a dyadic measure, which can be represented by a dyadic product formula representation. The product coefficients in the representation are unique. \textit [A Differential Privacy Property] Any two finite subsets of $E$ which determine measures with the same support (i.e. are positive on the same leaf sets), determine the same simplicial complexes $\mathcal{N}(\mu)$, $\mathcal{N}_0(\mu)$, and $\mathcal{N}_1(\mu)$, and thus have the same betti numbers for each of these simplicial complexes. 
 \end{enumerate}
\end{lem}
\begin{proof}
 Statements 1 and 2 are proved in Section \ref{dyadicsets}. Statement 3 is proved in Section \ref{nervecomplexes}. Statement 4 is true because equal support implies equality of maximal faces for each of the three simplicial complexes, which implies equality of the three simplicial complexes, which implies equality of betti numbers. Statement 5 is proved in Sections  \ref{dyadicsets} and \ref{pcs}. 
\end{proof}
The privacy labels in statements 4 and 5 of the theorem are justified because the simplicial complexes and their betti numbers do not reveal anything quantiative about the size of the sets assigned by the measures. They just reveal exploit the non-negativity of the measure on the support sets. The term differential privacy\cite{DwRo} is used because it applies if one of the samples contains just one additional point. The restriction for these privacy properties is that the support for the two measures (or samples) has to be the same.

\section {Data Quality Application Analysis}
\subsection{Data Quality Summary Statistics} 
Table~\ref{tab:cvsummary} summarizes the data quality constraint violations by source for each of the five sources. Recall the first source was the composite of the other four sources. There were 30 data quality constraints $F_i, i = 1 ..., 30$. A violation occurred if $F_i(data\ element) = 0$. 
The first column lists the maximum number of constraints violated by a data element.  The row corresponding to $i$ violations gives the number of data elements in each of the sources which violated a maximum of $i$ constraints; these numbers are listed in the columns labels $V^1, ..., V^5$ corresponding to sources 1, ..., 5. No data element violated more than 10 constraints, but every data item violated at least one data quality constraint.  The last row of the table lists the total number of elements in each source set.

\begin{table}[ht]
\caption{Summary of Mutual Constraint Violations by Source}
\centering
\begin{tabular}{c c c c c c}
\hline\hline
Number of Constraints Violated & $V^1$ & $V^2$ & $V^3$ & $V^4$ & $V^5$\\ [0.5ex] 
\hline
1    &    2487 &        385     &    797   &      465    &     840\\
           2   &      742    &      83 &        267   &      145  &       247\\
           3      &   255   &       26    &      98    &      52     &     79\\
           4   &      156   &       14    &      56     &     40      &    46\\
           5    &     165   &       15     &     49     &     43      &    58\\
           6     &     91    &       6      &    31     &     37  &        17\\
           7     &     18     &      1      &     6      &     7     &      4\\
           8     &      7     &      1     &      3      &     0     &      3\\
           9       &    1     &      0     &      0      &     0     &      1\\
          10     &      2      &     0   &        2     &      0     &      0\\[1ex]
\hline
Total Number of Data Elements & 3924   &      531    &    1309     &    789     &   1295\\[1ex]
\end{tabular}
\label{tab:cvsummary}
\end{table}

\subsection{Representation of Counting Measure for the Data Quality Data}

We ordered the 30 data quality constraints $F_1, ...., F_{30}$ in decreasing number of violations for the total data set  and found that 5 of the constraints were not violated. We then computed the ordered labeled binary tree for each of the five sources determined by the binary features $F_i$. 
The nodes of the tree correspond to mutual violation constraint sets and the leaves of the tree correspond to maximal mutual violation constraint sets. The counting measure for each of the nodes is the number of data items which violate the constraints specified by the path label for the mutual constraint set. The counting measure for the leaf nodes is the number of data items which have the data quality vector specified by the root to leaf path. We found that only 219 leaf nodes out of a potential $2^{30}$ leaf nodes had non-zero counting measure. In other words, the ~4000 data items determined only 219 distinct data quality binary feature vectors. We also found that 5 of the 30 data quality constraints were not violated by any data item. 
 If a node has counting measure 0, all of its descendants would also have counting measure 0, so only the non-empty part of the tree needed to be computed to obtain the non-zero product coefficient parameters characterizing the data quality counting measure. 
 This example points out the importance of computing the counting measure for binary feature data. The number of feature vectors here is many fewer than the number of data items; furthermore the data is not uniformly distributed over the binary feature vectors. 
 The counting measure can be represented both as a histogram using the discrete events determined by the unique data quality vectors  and in terms of the product coefficient parameters. We will first give summary information for the histogram and then will show visualizations for the the product coefficient parameter representation. 

\subsubsection {Histogram Representation} 
The histogram for the counting measure for the composite source, Source 1, has 219 distinct events, each corresponding to a unique data quality vector. (The total volume of this histogram would be the number of data items, since it is histogram for a measure which need not have total volume 1 rather than a histogram for a probability distribution which is required to have total volume 1).  Each of the 219 distinct events corresponds to a leaf in the binary tree for the source whose path label is the data quality vector. The histogram data for the non-composite sources 2,3, 4 and 5 is shown in tables Table~\ref{tab:source2}, Table~\ref{tab:source3-part3}, Table~\ref{tab:source4-part1}, Table~\ref{tab:source4-part2}, Table~\ref{tab:source5-part1} , Table~\ref{tab:source5-part2}. In each of these tables, there is one row for each data quality vector with non-zero counting measure in the source.  The rows are listed in decreasing order of violation for each source. The first column is the counting measure for the data quality vector for the source; the third column summarizes the data quality vector by listing the constraints which have value 0, i.e. violated; the second column enumerates the data quality vector. The composite source, Source 1, is obtained by combining the information in these tables. 

The histograms are long-tailed. For example, if the data quality vectors for the composite source are ranked by their counting measure, the first 52 data quality vectors in the ranking accounted for 90\% of the data  and the last 167 data quality vectors accounted for only 10\% of the data; the 52 data quality vectors involved 22 data quality constraints -numbered 1 ... 20, 22 and 23.   Analysis of the top 50\% of the histogram data for each source is shown in Table~\ref{tab:top50percent}. In the table the first column is the source number, the second column abbreviates a data quality vector by listing the numbers of the constraints whose value is 0 (not satisfied), and the third column lists the percentage of the data in the source with this data quality vector . The summary of the table is: the  top 50\% of the constraint violations are explained by 8 patterns involving 6 constraints. Analysis of the top 60\% of the histogram data for each source is shown in Table~\ref{tab:top60percent}. The summary of this table is: the top 60\% of violations are explained by 11 patterns involving 11 constraints.

\begin{table}[h!]
\caption{Top 50\% of Violations Explained by 8 Patterns Involving 6 Constraints}
\centering
\begin{tabular}{c r r}
\hline\hline
Source  & Constraints Violated  & Percentage  \\ [0.5ex] 
1 & 1 & 41 \\
1 & 4 & 6\\
1 & 2 & 4 \\
\hline
2 & 1 & 48\\
2 & 4 & 14\\
\hline
3 & 1 & 38\\
3 & 7 & 5\\
3 & 12 & 5\\
3 & 1, 12 & 4\\
\hline
4 & 1 & 32\\
4 & 6 & 15\\
4 & 1,6 & 8\\
\hline
5 & 1 & 46\\
5 & 4 & 7\\
\hline

\hline
\end{tabular}
\label{tab:top50percent}
\end{table}

\begin{table}[h!]
\caption{Top 60\% of Violations Explained by 11 Patterns Involving 11 Constraints}
\centering
\begin{tabular}{c r r}
\hline\hline
Source  & Constraints Violated  & Percentage  \\ [0.5ex] 
1 & 1 & 41 \\
1 & 4 & 6\\
1 & 2 & 4 \\
1 & 6 & 3\\
1 & 7 & 3\\
1 & 2,3,5,8,11& 3\\
\hline
2 & 1 & 48\\
2 & 4 & 14\\
\hline
3 & 1 & 38\\
3 & 7 & 5\\
3 & 12 & 5\\
3 & 1, 12 & 4\\
3 & 2 & 4\\
3 & 13 & 4\\
3 & 4 & 3\\
\hline
4 & 1 & 32\\
4 & 6 & 15\\
4 & 1,6 & 8\\
4 & 4 & 6\\
\hline
5 & 1 & 46\\
5 & 4 & 7\\
5 & 2 & 4\\
5 & 1,4 & 3\\
5 & 7 & 3\\
\hline

\hline
\end{tabular}
\label{tab:top60percent}
\end{table}

\subsubsection {Product Coefficient Representation and Visualization}

The product coefficient parameter representation provides more subtle information than the histogram and provides a set of statistics that could be used for multiscale statistical hypothesis testing (e.g. to determine at what scales there is a statistically significant difference between the data sources) and for statistical prediction. For each node $n$ in the binary tree, there is a product coefficient $a_n$. Assume the root node is at level 0 in the tree. For a node at level $i -1, i >= 1$, The path label  $L(n)$ for the node specifies a vector of values for the first $i -1$ data quality constraints and determines a set of data  which satisfies these constraint values. Let $C_{i} $ denote the $i$th constraint. This constraint splits the  set into two parts: the left part consists of data items which violate the $i^{th}$ constraint; the right part consists of data items
  which satisfy the $i$th constraint. The product coefficient for the node is  the difference between their relative proportions; It is a skewness measure. It is negative if more of the data set falls into the right half and positive if more of the data falls into the left half. Equivalently, it is  the difference between two conditional probabilities:

\begin{equation}
a_n = Pr(\neg C_{i}  | L(n)) - Pr(C_{i} | L(n))
\end {equation}

There are  a large number of product coefficients $2^{31} - 1$, since there 30 levels (and almost  all of them have the value 0 because there is no data in the corresponding constraint set intersection. They can be used to compare the data quality measure with any other measure on a universe with a binary tree structure. They can be used as features for decision algorithms; and a histogram of their values can be summarized. They can also be visualized using a daywheel figure. The daywheel visualizations of the  product coefficients for the first 12 levels are shown in Figures \ref{fig:source1} and \ref{fig:sources2to5}.
A daywheel visualizes the scale 0 product coefficient in the center (level 0) , the two scale 1 product coefficients in two halves of the first concentric ring, the four scale 2 product coefficients in four quarters of the second concentric ring, and the $2^{ith}$ product coefficients in the $i^{th}$ ring. The parent child structure of the binary tree is reflected in the parent child relationship between the divisions of neighboring concentric rings. The value of each product coefficient is color coded.  If the node set associated with the product coefficient is empty (i.e. there are no constraint violations) it is colored green. Red indicates a product coefficient value of -1, which means "skewed to violations"; purple indicates a product coefficient value of 0, which means the number of violations equals the number of non-violations; blue indicates a product coefficient vaue of -1 which indicates "skewed to non-violations". These three colors were encoded as the color vectors: red = $[1,0,0]$, blue = $[0,0,1]$, and purple = $[1,0,1]$. The product coefficient values were displayed by convex interpolation between these colors. The "sea of green" in each of the figures visualizes the large number of constraints set intersections which contain no data items. The tree structure of the green area corresponds to the fact that subsets of empty sets are empty. The red area visualizes the sets where there are relatively more violations than non-violations for the next constraint and the blue area visualizes the sets where there are relatively more non-violations than violations for the next constraint. The rays of red going out to the boundary show where the constraint violations are concentrated as the constraints are added one by one. Visually, Source 3 is most similar to the composite source, Source 1. 

\begin{figure}[h!]
\begin{center}
\includegraphics{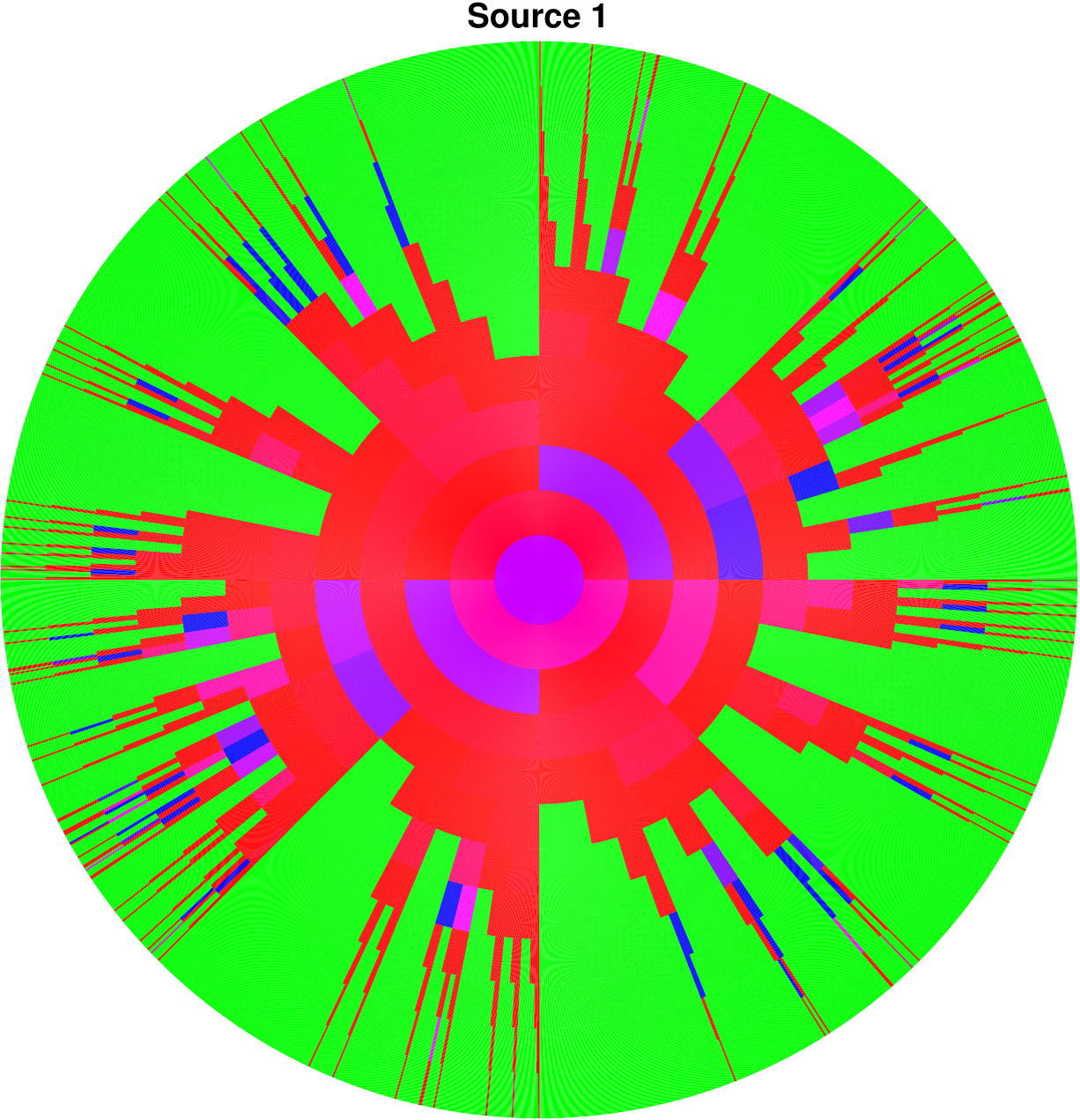}
\caption{Visualization of First 12 Constraints for Source 1 -- the Composite Source}
\label{fig:source1}
\end{center}
\end{figure}

\begin{figure}
    \centering
    \begin{subfigure}[b]{0.3\textwidth}
        \includegraphics[width=\textwidth]{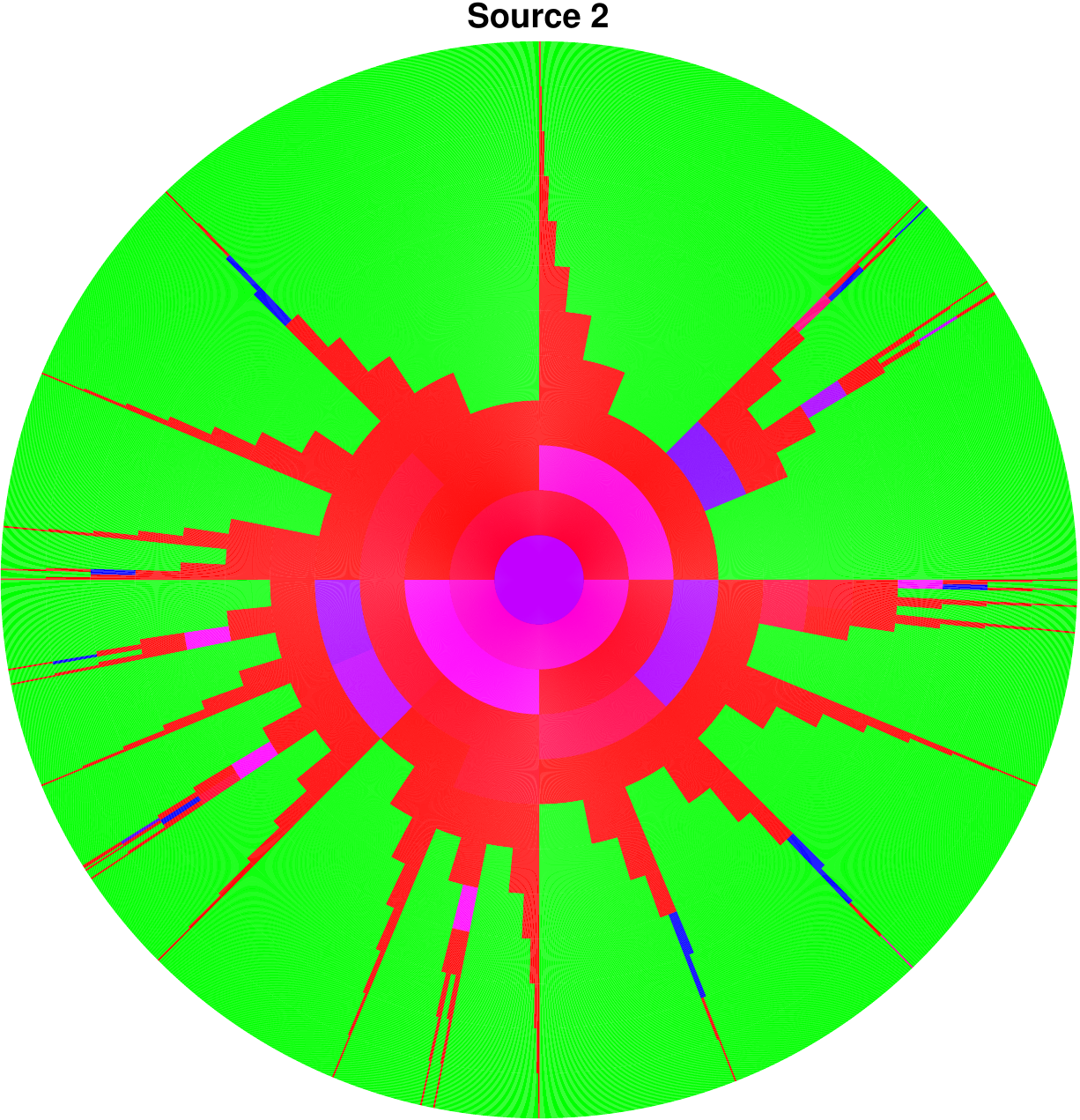}
        \caption{Source 2}
    \end{subfigure}
    \quad
    \begin{subfigure}[b]{0.3\textwidth}
        \includegraphics[width=\textwidth]{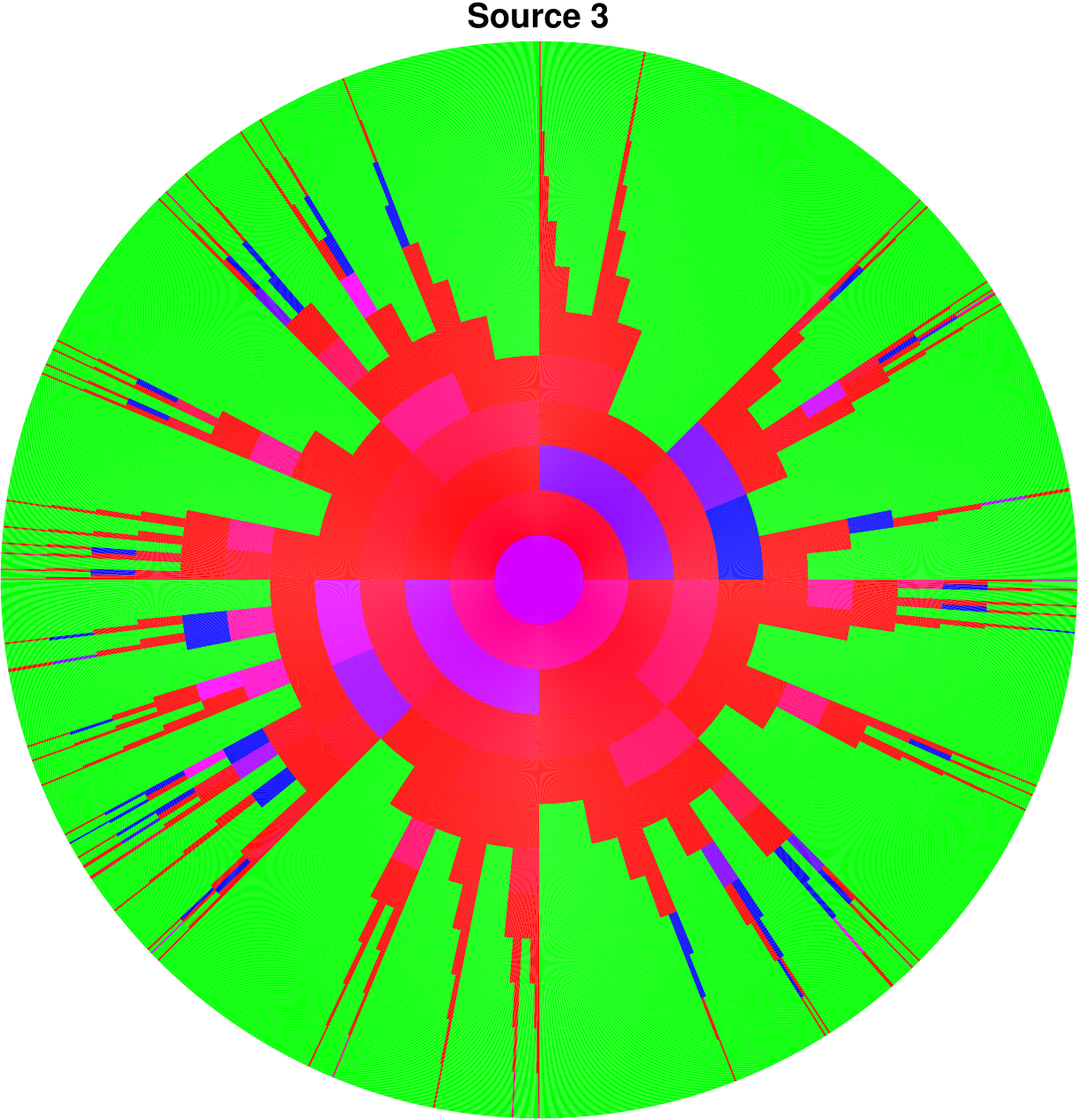}
        \caption{Source 3}
    \end{subfigure}

    \begin{subfigure}[b]{0.3\textwidth}
        \includegraphics[width=\textwidth]{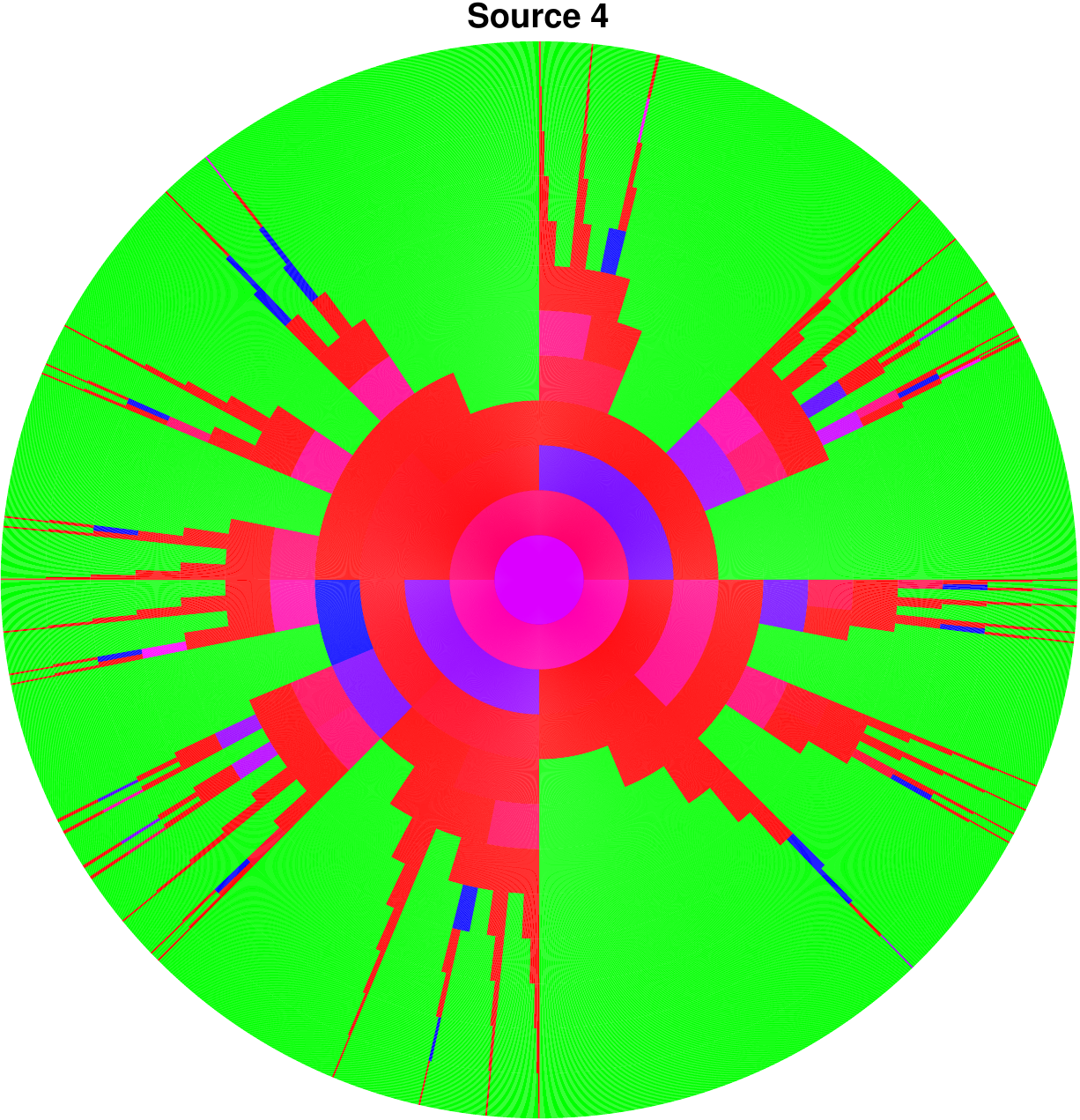}
        \caption{Source 4}
    \end{subfigure}
    \quad
    \begin{subfigure}[b]{0.3\textwidth}
        \includegraphics[width=\textwidth]{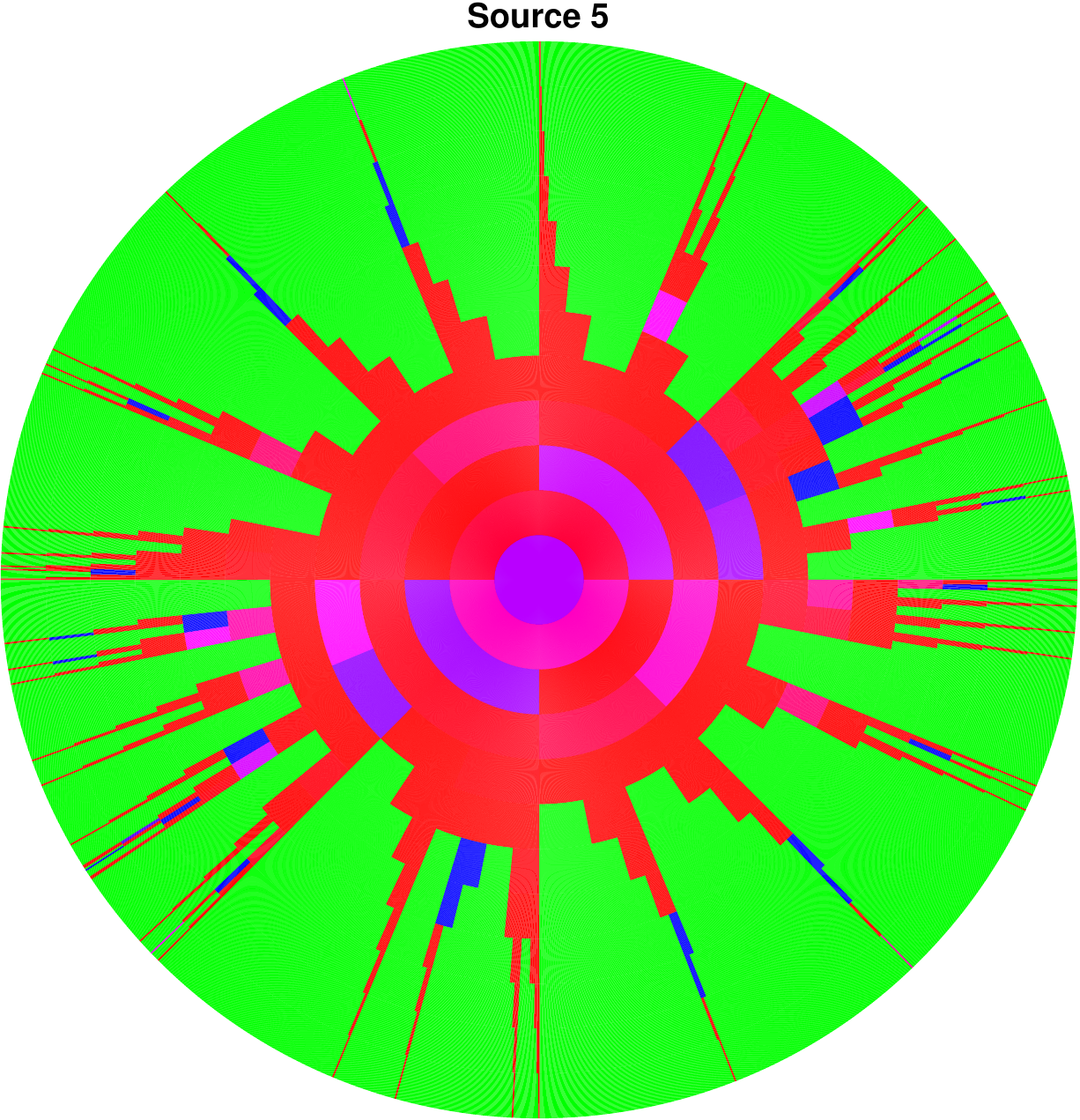}
        \caption{Source 5}
    \end{subfigure}
    \caption{Comparative Visualization of the First 12 Constraint Sets for Individual Sources}
    \label{fig:sources2to5}
\end{figure}

\subsection{Betti Numbers for the Data Quality Data}
The homological dimensions of the nerve simplicial complex $\mathcal{N}_{\mathcal{C_0}}$ determined by family $\mathcal{F} = \{C_1, ....., C_{30}\}$ of the 30 constraint violation sets were computed. Recall that a constraint violation set consists of the set of data items which violate the constraint. The computation of the betti numbers was done  by the author using the open source Computational Homology Project software (CHomP)\cite{HaMi2014}\cite{chomp}. 

\begin{table}[ht]
\caption{Homological Dimension Statistics for the Data Quality Sources - Betti numbers}
\centering
\begin{tabular}{c r r r r r r r r r r }
\hline\hline
Source Number & $\beta_0$ & $\beta_1$ &$ \beta_2$ &$ \beta_3$ & $\beta_4$ & $\beta_5$ & $\beta_6$ & $\beta_7$ & $\beta_8$ & $\beta_9$ \\ [0.5ex] 
\hline
\hline
1 & 2 & 2 &  7 & 0 & 0 & 0 & 0 & 0 & 0 & 0\\
2 &4 & 2 & 0 & 0 & 0 & 0 & 0 & 0 \\
3 & 2 & 8 & 3 &  0 & 0 & 0 & 0 & 0 & 0 & 0\\
4 & 2 & 1 & 2 & 0 & 0 & 0 & 0\\
5 & 1 & 2 & 0 & 0 & 0 & 0 & 0 & 0 & 0\\
\end{tabular}
\label{tab:bettisummary}
\end{table}

The number of homological dimensions for each source is the dimension of the largest face, or equivalently, the maximum number of constraints which were mutually violated in the source data, or the maximum number of zeros in a binary feature vector  representing the source data. This can easily be checked directly  from the binary feature representation data. This varied from set to set. Table ~\ref{tab:bettisummary} shows that the composite source (Source1) and Source 3 each had 10 simplicial dimensions. Sources 2, 4, and 5 had 8,7, and 9 simplicial dimensions respectively. The betti numbers summarize the simplicial geometry of the nerve complex for each source determined by the data quality constraint features. Only the first three betti numbers were non-zero for some source. The $0{th}$ betti number indicates the number of connected components. Sources 1, 3 and 4 each had two connected components; source 2 had 4 connected components and source 5 had only one connected components. The second source had only two non-zero betti numbers. The first betti number indicates the number of independent groups of 1 dimensional faces describing edges of triangles which are unrealized as two-dimensional simplices; the first betti numbers for sources 1, ..., 5 are 2, 2, 8, 1 and 2. The second betti number indicates the number of independent groups of 2 dimensional faces corresponding to faces of tetrahedron which are unrealized as three-dimensional simplices; the second betti numbers for sources 1, ..., 5 are 7, 0, 3, 2, and 0 respectively. Already the second betti number indicates relatively complex simplicial geometry exists for these constraint violation features sets. 
The betti numbers are a new type of integer-valued statistics that can be computed for binary feature data. In this case, Table ~\ref{tab:bettisummary} shows that these statistics distinguish the source data sets. 

\section{Related Work and some Additional Research Directions}

\subsection{Measures on Tree Structured Spaces}

The dyadic product formula representation was first made explicit  for the unit interval in the 1991 Annals of Math paper ``The Theory of Weights and the Dirichlet Problem for Elliptic Equations'' authored by R. Fefferman, C. Kenig and J. Pipher \cite{FeKePi91}  in Section 3.18, Lemma 3.20. (Weights are positive functions.) In this paper the authors were trying to prove that certain weights (``harmonic measures'') arising in elliptic PDE lie in the Coifman Fefferman class $A_\infty$. They noted that the $A_\infty$ condition holds if and only if the measure is doubling and a certain $L^2$ condition is satisfied for the coefficients. They used it to construct an explicit example of a doubling measure on the unit interval, which provided a counterexample necessary to a  proof of one of their main theorems. The restricted set of dyadic measures that they used is very far from the case of general measures and in particular the $L^2$ condition they studied does not hold for multifractal measures which typically arise in for many real data sets.
 In \cite{BaNeJoSh16} the author and collaborators realized the result could be re-formulated for dyadic measures on sets with binary tree structures (not just the unit interval) and used to provide an algorithmizable theoretically-based  method for representing finite data samples from universes with a binary tree structure as vectors of product coefficient parameters of measures. They exploited the fact that the weak star limit result holds for the much more general class of non-negative dyadic measures.  In the paper, they also formulated a visualization theorem and a multscale noise theorem for strictly positive measures that were re-formulations of deep results. Their visualization theorem exploited deep results due to Beurling and Ahlfors \cite{BeurlingAhlfors56} and Ahlfors \cite{Ah66} from the the theory of quasi-conformal mapping. The results enable construction of a Jordan plane curve (a welding curve) from a measure on the unit circle satisfying mild constraints on its product coefficient parameters and characterize its uniqueness. They exploited the multiplicative model of chaos defined by Kahane\cite{Ka85} \cite{RhVa14} to define a multi
scale noise model for dyadic measures and exploited the analysis in Kahane's proof to obtain noisy measures with finite, non-zero volume. 
The noise model is related to Brownian motion. Recent work by Grebenkov, Beliaev and Jones \cite{GrBeJo} provides an exposition of L\'{e}vy's formulation of Brownian motion in terms of explicit dyadic multi scale formalisms. They revise  L\'{e}vy's construction of Brownian motion to operate with various Gaussian processes. A Brownian path is explicitly constructed as a linear combination of dyadic ``geometrical features'' at multiple length scales with random weights. Such a representation gives a closed formula mapping of the unit interval onto the functional space of Brownian paths.

In \cite{KoNo04}  Kolaczyk and Nowak developed a systematic approach to multiscale probability models. They showed that multiscale factorizations, similar to Lemma 3.20 in \cite{FeKePi91}  and the refomulation of it for dyadic measures on binary tree structured spaces, Lemma 2.1 in \cite{BaNeJoSh16}, arise when conditions for a "multiresolutionanalysis (MRA)" of likelihoods are satisfied  and shown that these conditions characterize the Gaussian, Poisson and multinomial models.  They also quantified the risk behavior of certain nonparametric, complexity penalized likelihood estimators based on their factorizations. They focused on binary tree structures. Our approach does not rely on such assumptions and we do not not provide theoretical results for risk estimation. 

A natural question to ask is: does the representation lemma for dyadic measures discussed in Section \ref{dyadicmeasures} and proven in \cite{BaNeJoSh16} and \cite{FeKePi91} generalize to measures on sets with a general tree structures?  We provide one answer in the next Lemma.  Fix notions by defining a tree structure on a set $X$ that consists of a tree $\mathcal{T}$, a mapping $\mathcal{S}: nodes(\mathcal{T}) \rightarrow 2^X$ from nodes to subsets of $X$ and constraints on the mapping: $\mathcal{S}(root) = X$ and the image of a parent node is the disjoint union of the images of its child nodes i.e.  $  S(n) = \cup_{c: child(n)} S(c)$

\begin{lem} [Representation Lemma for Measures on Tree Structured Sets]
Let $ (X, \mathcal{T}, \mathcal{S})$ denote a tree-structured set. Let $\nu$, $\mu$ denote strictly positive and non-negative measures, respectively, on the sigma algebra $\Sigma(\mathcal{S})$ generated by the sets in the image of $\mathcal{S}$. For a non-root node $n \in \mathcal{T}$ let $r >= p >= n$ denote the set of nodes $p$ on the path from $n$ to the root node $r$ ordered by the parent relationship. Let $a_n$ be the parameter uniquely defined by 
\begin{equation}
\mu(S(n)) = (1 + a_n) \frac{\nu(S(n))}{\nu(S(parent(n)))} \mu(S(parent(n)))
\end{equation}
if $\mu(S(parent(n))) \neq 0$. If $\mu(S(parent(n))) = 0$ define $a_n = 0$. 
\begin{itemize}
\item $\nu(S(n)) = \nu(X) \prod_{r >  p >= n} {\frac{\nu(S(p))}{\nu(S(parent(p))}}$
\item $\mu(S(n)) = \mu(X) \prod_{r > p >= n} { (1 + a_p) \frac{\nu(S(p))}{\nu(S(parent(p))}}$
\item for a non-leaf node $p$, the function $f: \{S(c): c \in child(p)\} \rightarrow \mathbb{R} , f(S(c)) = a_c \cdot \nu(S(c))\}$ is orthogonal to the functions which are constant on the $\{S(c): c \in child(p)\}$. 
\begin{equation}
0 = \sum_{c: child(p)} {a_c \cdot \nu(S(c))}
\end{equation}
Thus, relative to a choice of a multi-resolution basis for functions on tree consisting of the parent node $n$ and its child nodes, the function $f$ can be expressed as a unique linear combination of $card(\{S(c): c \in child(p)\}) - 1$ basis functions. For binary trees, the function $h_p$ which is 1 on the left child node $c_L$ and -1 on the right child node $c_R$  is a basis function, and the linear combination is $a_{c_L} \cdot h_p$.  
\item   $S$ maps the set of nodes at each level $i$ in the tree to disjoint partitions $\mathcal{P}_i$ of $X$. The partition at level  $i+1$ refines the partition at level $i$. Let $\nu_i$ and $\mu_i$ denote the measures on $\Sigma(\mathcal{P}_i)$, the sigma algebra generated by the sets in $\mathcal{P}_i$, determined by the first two path formulas. The the weak star limit of $\nu_i$ and $\mu_i$ exists.  
\item $-1 <= a_n < \frac{\nu(S(parent(n))) } {\nu(S(n))} - 1 $
\end{itemize}
\end{lem}

\begin{proof}
The first statement, the path formula for $\nu$, is trivially true by telescoping cancellation. Informally, it just says that the $\nu(S(n))$ equals  $\nu(X)$ multiplied by the conditional probabilities determined by a path from the root node to $n$.  None of the denominators in the conditional probabilities are zero because $\nu$ is strictly positive. The second statement, the path formula for $mu$, is proved by induction on the length of the path. It is true for a path of length 1 beginning at the root, by the definition of the parameter in the statement of the lemma. The induction step just substitutes the induction hypothesis for  $\mu(S(parent(n)))$. The third statement is clearly true if $\mu(S(p)) = 0 $ because then the parameters $a_c$ for the child nodes are all 0. If $\mu(S(p)) \neq 0 $
the third statement is proved by noting that for a non-leaf node $\mu(S(p)) = \sum_{c: child(p)} {\mu(c))}$ because measures are additive on disjoint sets and the tree structure definition requires that the disjoint union of the set images of child nodes equals the set image of the parent node. Substituting the path formula for $\mu$ into the right side,  expanding, and factoring out $\mu(S(p))$  gives
\begin {equation}
\mu(S(p)) = \mu(S(p)) \cdot (\sum_{c: child(p)} {\frac{\nu(S(c))}{\nu(S(p))} } + \sum_{c: child(S(p))} {a_p \cdot \frac{\nu(S(c))}{\nu(S(p))} } )
\end{equation}
The first sum term in parentheses is 1 because the sum of the $\nu$ conditional probabilities of the child nodes equals 1. The the second sum term in parentheses equals 0. Multiple by $\nu(p)$ to obtain the third statement. The first two sentences in the fourth statement are implied by the definition of tree structure. The fourth statement claim that the weak star limit exists is proved using the same method as used for  Lemma 2.1 in \cite{BaNeJoSh16} and Lemma 3.20 in  \cite{FeKePi91}. The key point is that $\nu_i(X) $ and $\mu_i(X)$ are constant for all levels $i$. The fifth statement is true if $\mu(S(parent(n))) = 0$ since then $a_n = 0$ by definition. If $\mu(S(parent(n))) \neq 0$, the definition can be rewritten as: 
\begin{equation}
 \frac{\mu(S(n))}{ \mu(S(parent(n)))} \cdot \frac{\nu(S(parent(n))) } {\nu(S(n))} = 1 + a_n
\end{equation}

Since $\nu$ is strictly positive and $\mu$ is non-negative the left side of the equation is non-negative, so $1 + a_n >= 0$ implying $a_n >= -1 $, with $a_n = -1$ only if $mu(S(n)) = 0$ and $\mu(S(parent(n))) \neq 0$. If $\mu(S(n)) > 0$ and $\mu(S(parent(n))) \neq 0$, $1 + a_n$ is biggest when $\mu(S(n)) = \mu(S(parent(n))$ (making the first factor in the equation equal to 1) so the result follows. If the tree is a regular k-adic tree and $\nu$ is the naive measure($\nu(X) = 1$ and $\nu(c) = \frac {1}{|child\; nodes|}$  $\frac{\nu(S(parent(n))) } {\nu(S(n))}$ is constant for all nodes and equals $k$ so $a_n = k - 1$ and is independent of $n$. This agrees with the dyadic theory where $k = 2$. 
\end{proof}

Statements 3 and 5  in the lemma shows that the parameter space for general tree measures is much more complex and depends node by node on the measure $\nu$. Statement 5 explains why the parameter space for dyadic measures on binary trees is simpler and easily related to the Haar-like basis. While it is easy to define a dyadic version of Wasserstein distance between dyadic measures,  definition of a canonical Wasserstein-like distance between measures on tree-structured sets appears to be a research issue. Perhaps the tree distance theory developed by Billera, Holmes and Vogtman \cite{BiHoVo2001} could be exploited. 

\subsection {Comparison with Multi-Resolution Theory}

Meng in \cite{ Meng2014}  has proposed the statistical issues posed by multi-resolution  as an important statistical research area.  Here we focus on a comparison with multi-resolution theory as recently described in a recent paper by Gavish, Nadler and Coifman \cite{GaNaCo}.
Gavish, Nadler and Coifman focused on developing a multi-resolution theory for functions on discrete data sets given a hierarchical tree structure.  Their point of view is multi-resolution analysis of function spaces and development of a theory for sparsity estimation. They use the given hierarchical structure to construct a measure determined by the data, similar to our approach. The nodes of the tree determine subsets of the data sets which are descendants of the nodes and the measure for each node set (which they call a folder) is the proportion of the data descended from the node. It is the normalization into a probability measure of the counting measure determined by the node. They use this measure to construct a tree metric. The distance between two data points is the measure of the smallest node containing both as descendants.  
 They also exploit the hierarchical tree structure to construct a wavelet-like orthonormal basis for the functions on the data set. The levels of the hierarchy determine the multi-resolution structure of  functions on the data. 
They then formulate a function smoothness criterion defined in terms of the tree metric determined from the probability measure and prove that function smoothness could be measured by the rate of function coefficient decay, which they  use to quantify the smoothness of a data set.  They showed exponential decay for tree structures which are balanced relative to the data and use it to prove existence of sparse approximations so only a few coefficients of a function (e.g. a classification function on the data) would have to be computed. A balanced tree structure implies that the data is "everywhere"  relative to the different levels of the tree structure, i.e. that the counting probability measure is strictly positive for all nodes in a bounded sense. The fast coefficient decay is used to prove existence of sparse approximations. 

Our approach would view the data as non-negative (not strictly positive) measure on a hierarchical tree structure. However, the hierarchical structure can be arbitrary and need not be determined by the data, but in practice is a natural one for the universe from which the data is sampled. In the paper we focus on using binary tree structures determined by binary feature sets appropriate for an application or set of data sample. The theory shows that such a measure is uniquely characterized by $a_n$ parameters (one for each node in the tree). The parameters determine a function approximating the data, which is a "radon nikodym derivative"  relative to a strictly positive measure which assumes even distribution among the children of a node. This is analogous to a Bayesian assumption. We would use the canonical parameters (the $a_n$ parameters) for the measure as features  characterizing the data set. The support of the representing measure is typically sparse because real data sets tend to be concentrated in parts of the tree so only a very small percentage of the parameters are needed to characterize the data set. Our experience shows that real data sets are not balanced with respect to natural hierarchical tree structures  and the measures determined by data sets are non-negative not positive (as required for the balanced assumption and as is traditionally assumed in statistics). We characterize sparsity by characterizing the support of the measure representing the data set in two ways - the scale and locality of the node sets in the tree with positive measure and the betti numbers for the simplicial complex representing the support of the measure. The characterization of the support provides a privacy property for samples of data with the same support.
Instead of constructing sparse representation of classifications functions, using a priori sparsity estimates, our approach is to use the parameters of the characterizing measure and the betti numbers as automatically generated features for the data so we can then we use traditional algorithms for decision making, e.g. decision trees \cite{ Ness16} exploiting the hierarchical structure of the parameters. 
The data quality examples informally illustrates that both the measure parameters and the betti numbers were sufficient to distinguish the data sets, illustrating their utility for semi-supervised learning. An additional relationship with multi-resolution theory is shown in  Statement 3 of the Representation Lemma for Measures on Tree Structured Sets.   Statement 3 shows that the volume function for the $a_c$ parameters for  children of a node $n$ can be represented relative to a chosen multi-resolution basis for the simple parent-child tree. Choices of these bases for each node (or each type of node) will result in a product formula similar to the product formula  similar to Lemma 3.20 in \cite{FeKePi91}  and Formula \ref{dyadicproductformula}. While for a general tree structure there is no single canonical basis, for the binary tree case only one basis element is needed so there effectively is a canonical choice as illustrated in Section \ref{pcs}.

\subsection {Computational Topology} 

Computational topology research is a very active area. One survey of the area and its research results is \cite{EdHa}. An example of research research by statisticians in the area is \cite{FaLe14}. It has 
developed the multiscale theory of persistent homology and a number of publicly available algorithms for computing it.  A key first step is to represent the data as a simplicial complex (e.g. the RIPS complex for metric space data). The resulting persistent homology dimensions are then statistics which can be computed algorithmically from the representation, represented using diagrams such as the persistence diagram  and analyzed further using special purpose statistics developed for persistence diagrams.  

We defined a variant of the fundamental nerve simplicial complex defined using a dyadic measure and its binary features, used it to represent the support of a dyadic measure, and then exploited basic computational topology algorithms to compute the betti numbers of the simplicial complex. We realized that these betti number statistics provide a type of differential privacy property for samples of the dyadic measure. The simplicial complexes we used have a filtration determined by the order of the binary features, so it would be natural to apply persistent homology algorithms to these filtered simplicial complexes to generate a richer set of  persistent homology statistics, which would also have some refined privacy properties.  Our approach did not require any assumption of metric space properties for the data, which are not always available. It seems possible that this approach would provide a new set of applications for persistent homology and establish more links between data, measure theory,  and topology. 

\subsection {Algebraic Statistics} 

Algebraic statistics\cite{DrStSu} may provide additional useful insights. The representations used here are fundamentally algebraic. For each scale $s$, the Dyadic Product Formula Representation (Lemma 2.1 in \cite{BaNeJoSh16} provides a product formula for the measure of each scale $s$ dyadic interval which is a polynomial of degree $s$ in the product coefficient parameters associated with ancestor nodes  The product coefficient parameters are structured as a tree. The product formula for scale $s+1$ adds a factor with an additional parameter for each leaf node. The information dimensions, entropy and variance for each scale $s$ measure are computed by algebraic formulas. The parameters are in $[-1,1]^{2^{s_1} - 1}$. Algebraic topology was used to characterize the support of the measures. The Bayes rule for computing product coefficient parameters after permuting the order of features is also a rational algebraic formula. 
The Representation Lemma for Measures on Tree Structured Sets implies a similar formula for measures on regular trees with nodes of higher degree. 

\section {Summary}
We proposed a new method for representing data sets for which a set of binary features have been defined which consists of computing both the dyadic set structure determined by an order on the binary features together with the canonical product coefficient parameters for the associated dyadic measure and a variant of a nerve simplicial complex determined by the support of the dyadic measure together with its betti numbers. The product coefficient parameters characterize the relative skewness of the dyadic measure at dyadic scales and localities enabling fusion, statistical prediction and multiscale hypothesis testing. The more abstract betti number statistics summarize the simplicial geometry of the support of the measure and satisfy a differential privacy property. Both types of statistics can be computed algorithmically from the binary feature representation of the data. This new representation provides a new theoretically-based method for pre-processing  data -- permitting new theoretically-based methods of  analysis -- e.g. new methods for statistical fusion, decision-making, inference, hypothesis testing and visualization. We illustrated the methods on a a data quality data set. 

We also proved that dyadic sets with dyadic measures have a canonical set of binary features and determine canonical nerve simplicial complexes. We compared our methods with some other results for measures on sets with tree structures, recent multi-resolution theory, and persistent homology and suggested links to differential privacy, Bayesian reasoning and algebraic statistics. More theoretically-based methods for representation of data will enable a richer set of methods for reasoning about data.  With such pre-processing methods, data sets  are transformed into examples of well-known mathematical structures enabling reproducible mathematical and statistical reasoning, which can be more easily compared and validated.

 \appendix

\section {Mutual Constraint Violations by Source}
The data in the tables in this appendix  is sufficient to reproduce the data quality analysis described in the paper. The input to this analysis were data quality statistics for each source describing the constraint violations for each data item in each source. The sources were the individual sources 2,3,4 and 5 and the composite source 1. The mutual constraint violations view of the data for each of the individual sources is shown in tables. The data for the composite source is a composite of the data for the individual tables, so that data is not shown in a table. The raw input spreadsheet tables were pre-processed into 3 column tables. In each of these tables, there is one row for each unique maximal set of constraints violated by a data element. The rows are listed in decreasing order of violation. The first column lists the number of elements whose maximal violation set is the one listed in the row. The third column lists the numbers of the constraints in the maximal violation set. The second column is the label of the path from the roots of the binary tree to the node corresponding to the maximal set. 
The mutual constraint violations for Source 2 are shown in Table~\ref{tab:source2}. The constraint violations for Source 3 are shown in Table~\ref{tab:source3-part1}, Table~\ref{tab:source3-part2} and Table~\ref{tab:source3-part3}. The constraint violations for Source 4 are shown in Table~\ref{tab:source4-part1} NS Table~\ref{tab:source4-part2}. The constraint violations for Source 5 are shown in Table~\ref{tab:source5-part1} NS Table~\ref{tab:source5-part2}.

\begin{table}[ht]
\caption{Mutual Constraint Violations for Source 2 in Decreasing Order of Violation}
\centering
\begin{tabular}{c l l}
\hline\hline
Number & Binary ID for Violation Set & IDs of Violated Constraints \\ [0.5ex] 

\hline
257 &   011111111111111111111111111111 &    1 \\
    72 &   111011111111111111111111111111 &     4  \\
    38 &   101111111111111111111111111111 &   2 \\
    19 &   001111111111111111111111111111 &   1,  2 \\
    18 &   011011111111111111111111111111 &   1, 4  \\
    15 &    111111110011111111111111111111 &    9, 10 \\
    14 &   100111111111111111111111111111 &    2, 3 \\
    11 &   100101111111111111111111111111 &    2, 3, 5 \\
    10 &    100101101101111111111111111111 &   2, 3, 5, 8, 11 \\
    9 &   011111111111101111111111111111 &    1, 14 \\
     5 &    111111111111110111111111111111 &   15 \\
    4 &   000101101101111111111111111111 &    1, 2, 3, 5, 8, 11\\
    4 &   000101111111111111111111111111 &    1, 2, 3, 5 \\
     4 &   110111110011111111111111111111 &   3, 9, 10 \\
    4 &   111110111111111111111111111111 &   6 \\
    4 &   111111111111101111111111111111 &   4\\
     3 &    000101101111111111111111111111 &     1, 2, 3, 5, 8 \\
    2 &   000111111111111111111111111111 &   1, 2, 3\\
     2 &  011110111111111111111111111111 &   1,6 \\
    2   &  011111110011111111111111111111 &    1, 9, 10\\
    2 &    011111111111111000111111111111 &   1, 16, 17, 18\\
     2 &   100101101111111111111111111111 &   2, 3, 5, 8 \\
     2 &    101011111111111111111111111111 &    2, 4\\
    2 &    110111110010111111111111111111 &   3, 9, 10, 12 \\
     2 &   111111111111111001111111111111 &   16, 17 \\
    2 &   111111111111111111101111111111 &   20\\
    1 &   000101101101111011111111111111 &   1, 2, 3, 5, 8, 11, 16\\
    1 &   000111110010111111111111111111 &   1, 2, 3, 9, 10, 12\\
    1 &   000111111111101111111111111111 &   1, 2, 3, 14 \\
    1 &    001111111111101111111111111111 &   1, 2, 14\\
    1 &   010111110011111111111111111111 &   1, 3, 9, 10\\
    1 &   011011111111101111111111111111 &   1, 4, 14\\
    1 &   011111111111111110111111111111 &    1, 18\\
    1 &   100001101101111111111111111111 &   2, 3, 4, 5, 8, 11\\
    1 &   100001111111111111111111111111 &   2, 3, 4, 5\\
    1 &   100011111111111111111111111111 &   2, 3, 4\\
    1 &   100101110011111111111111111111 &   2, 3, 5, 9, 10\\
    1 &   100111111111111011111111111111 &  2, 3, 16\\
    1 &   101101101111010111111111111111 &   2, 5, 8, 13, 15 \\
    1 &   101101111111111111111111111111 &    2, 5 \\
    1 &   101111111111111001111111111111 &    2, 16, 17\\
    1 &   110011110011111111111111111111 &   3, 4, 9, 10\\
    1 &   111111011111111111111111111111 &   7\\
     1 &   111111110010111000110101111111 &   9,10,12, 16, 17, 18, 21, 23\\
    1 &  111111110010111111111111111111 &    9, 10, 12\\
    1 &   111111111111011111111111111111  &   13\\
    1 &   111111111111111000111111111111 &    16, 17, 18\\
    1 &   111111111111111111111110111111 &  24\\
\hline
\end{tabular}
\label{tab:source2}
\end{table}

\begin{table}[ht]
\caption{Mutual Constraint Violations for Source 3 in Decreasing Order of Violation Part 1}
\centering
\begin{tabular}{c l l}
\hline\hline
Number & Binary ID for Violation Set & IDs of Violated Constraints \\ [0.5ex] 
\hline

502  &  011111111111111111111111111111 & 1\\
     66 &   111111011111111111111111111111 &    7\\
    63 &    111111111110111111111111111111 &   12\\
     53  &  011111111110111111111111111111 &  1,12\\
     49 &   101111111111111111111111111111 &   2\\
    48 &   111111111111011111111111111111 &   13\\
    41 &    111011111111111111111111111111 &    4 \\
    33 &   100101101101111111111111111111 &    2, 3, 5, 8, 11\\
    28 &   011111011111111111111111111111 &   1, 7\\
    28 &   111111110011111111111111111111  &  9, 10\\
    22 &  100111111111111111111111111111 &   2, 3\\
    18 &  000101101101111111111111111111 &   1, 2, 3, 5, 8, 11\\
    17 &   011111110011111111111111111111 &   1, 9, 10\\
    16   & 001111111111111111111111111111 &   1, 2 \\
    16 &   011011111111111111111111111111 &    1,4 \\
    16 &    011111111111101111111111111111 &    1,14\\
    15 &    100101111111111111111111111111 &    2, 3, 5\\
    14 &    111011011111111111111111111111 &   4, 7\\
    13 &   000101111111111111111111111111   & 1, 2, 3, 5\\
    13 &    111111111111010111111111111111 &    13, 15\\
    10 &  110111110011111111111111111111 &    3, 9, 10\\
    9 &   011111111111011111111111111111 &    1, 13\\
    8 & 000111111111111111111111111111 &   1, 2, 3 \\
    8 &  111111111110111000111111111111 &    12, 16, 17, 18\\
    7 &    111111011111110111111111111111 &    7,15\\
    6 &  011011011111111111111111111111 &   1, 4, 7\\
    6 &  011111111111111011111111111111 &   1, 16\\
    6 &   101111111110111111111111111111 &    2, 12\\
    6 &  111111111111101111111111111111 &    14\\
    5 &  111011111111011111111111111111 &    4, 13\\
    5 &   111111110011111111011111111111 &    9, 10, 19\\
    5 &   111111111111110111111111111111   & 15\\
    5 &   111111111111111011111111111111 &    16\\
    4 &   011111011110111111111111111111 &   1, 7,12\\
    4 &   100101101111111111111111111111 &  2, 3, 5, 8\\
    4 &   111111011110111111111111111111 &   7, 12\\
    4 &  111111110011111111010101111111 &    9, 10, 19, 21, 23\\
    4 &    111111111110011111111111111111 &   12, 13\\
    4 &    111111111111111111101111111111 &    20\\
    3 &   011011111111111011111111111111 &    1, 4, 16\\
    3 &  011110111111111111111111111111 &   1, 6\\
    3 &   100011111111111111111111111111 &   2, 3, 4\\
    3 &   101111011111111111111111111111 &   2, 7\\
    3 &  110101101101111111111111111111 &    3, 5, 8, 11\\
    3 &   110111111111111111111111111111 &   3\\
    3 &  111011111110111111111111111111 &    4,12\\
    3 &  111111010011111111111111111111 &   7, 9,10\\
    3 &   111111110010111111111111111111 &   9,10,12\\
\hline
\end{tabular}
\label{tab:source3-part1}
\end{table} 

\begin{table}[ht]
\caption{Mutual Constraint Violations for Source 3 in Decreasing Order of Violation Part 2}
\centering
\begin{tabular}{c l l}
\hline\hline
Number & Binary ID for Violation Set & IDs of Violated Constraints \\ [0.5ex] 
\hline

    2 &   000001101101111111111111111111 &    1, 2, 3, 4, 5, 8,11\\
    2 &   000111111110111111111111111111 &    1, 2, 3, 12\\
    2 &    010111110011111111111111111111 &    1, 3, 9, 10\\
    2 &   011111010011111111111111111111 &    1, 7, 9, 10\\
    2 &   011111110010111111111111111111  &  1, 9, 10, 12\\
    2 &  011111110011111111010101111111 &    1, 9, 10, 19, 21, 23\\
    2 &   011111110011111111011111111111 &    1, 9, 10,19\\
    2 &  011111111110111000111111111111 &    1,12,16,17,18\\
    2 &  100001001101111111111111111111 &   2, 3, 4, 5, 7, 8,11\\
    2 &   100001101101111111111111111111 &   2, 3, 4, 5, 8,11\\
    2 &   100101100001111111010101111111 &     2, 3, 5, 8, 9, 10, 11, 19, 21, 23\\
    2 &   100101111101111111111111111111 &    2, 3, 5, 11\\
    2 &    100101111110111111111111111111 &   2, 3, 5, 12\\
    2 &    100111111110111111111111111111 &    2, 3,12\\
    2 &   101011111111111111111111111111 &  2, 4\\
    2 &   110111110010111111111111111111 &  3, 9, 10, 12\\
    2 &  111110111110111111111111111111 &    6, 12\\
    2 &   111111111111111000111111111111 &    16, 17, 18\\
    2 &    111111111111111111111011111111 &   22\\
    1 & 000001101111111111111111111111 &   1, 2, 3, 4, 5, 8\\
    1 &   000101011111111111111111111111 &   1, 2, 3, 5, 7\\
    1 &   000101101110111001111111111111 &  1, 2, 3, 5, 8, 12, 16, 17\\
    1 &   000101101111111111111111111111 &   1, 2, 3, 5, 8\\
    1 &   000101110011111111111111111111 &    1, 2, 3, 5, 9, 10\\
    1 &   000111110011111111111111111111 &   1, 2, 3, 9,10\\
    1 &   001101111111111111111111111111 &   1, 2, 5\\
    1 &   001111111110111111111111111111 &  1, 2, 12\\
    1 &   010011110011111111111111111111 &   1, 3, 4, 9, 10\\
    1 & 010101100011111111010111111111 &   1, 3, 5, 8, 9, 10, 19, 21\\
    1 &  010101111111111111111111111111 &   1, 3, 5\\
    1 &  010111010011111111111111111111 &   1, 3, 7, 9, 10\\
    1 &   010111110010111111101111111111 &    1, 3, 9, 10, 12, 20\\
    1 &  010111110011011111111111111111 &   1, 3, 9, 10 , 13\\
    1 &   010111111111111111111111111111 &   1, 3\\
    1 &   011011010011111111111111111111   & 1, 4, 7, 9, 10\\
     1 &   011011011111101111111111111111 &   1, 4, 7, 14\\
    1  &  011011110011111111111111111111  &  1, 4, 9, 10\\
    1 &   011011111111101111111111111111 &   1, 4, 14\\
    1 &   011110011111111111111111111111 &   1, 6, 7\\
    1 &   011111011111110111111111111111 &   1, 7, 15\\
    1 &  011111110011011111111111111111 &     1, 9, 10, 13\\
     1 &   011111111110011111111111111111 &    1, 12, 13\\
     1 & 011111111111101011111111111111 &    1, 14, 16\\
    1 &  011111111111111000111111111111 &   1, 16, 17, 18\\
    1 &  011111111111111111111011111111 &  1, 22\\
    1 &   011111111111111111111110111111 &    1, 24\\
    1 &   100001101101011111111111111111 &  2, 3, 4, 5, 8, 11, 13\\
    1 &   100001101111111111101111111111 &   2, 3, 4, 5, 8, 20\\
    
\hline
\end{tabular}
\label{tab:source3-part2}
\end{table}

\begin{table}[ht]
\caption{Mutual Constraint Violations for Source 3 in Decreasing Order of Violation Part 3}
\centering
\begin{tabular}{c l l}
\hline\hline
Number & Binary ID for Violation Set & IDs of Violated Constraints \\ [0.5ex] 
\hline
    
    1 &   100011001101111111111111111111 &    2, 3, 4, 7, 8, 11\\
    1 &   100011011111111111111111111111 &   2, 3, 4, 7\\
    1 &   100101000001111111111111111111 &    2, 3, 5, 7, 8, 9, 10, 11\\
    1 &    100101001100111111111111111111 &   2, 3, 5, 7, 8, 11, 12\\
    1 &   100110101111111111111111111111 &    2, 3, 6, 8\\
    1 &   100111101101111111111111111111     & 2, 3, 8, 11\\
    1 &   100111110010111111111111111111 &    2, 3, 9, 10, 12\\
    1 &   100111110011111111111111111111 &   2, 3, 9, 10\\
    1 &   101011011111111111111111111111 &    2, 4, 7\\
    1 &    101011111110111000111111111111 &   2, 4, 12, 16, 17, 18\\
    1 &  101011111111011111111111111111 &    2, 4, 13\\
    1 &  101101111111111111111011111111 &   2, 5, 22\\
    1 &    101111011110111111111111111111 &    2, 7, 12\\
    1 &  101111111110111000111111111111 &    2, 12, 16, 17, 18\\
    1 &   110011110011111111111111111111 &   3, 4, 9, 10\\
    1 &  110101110011111111011111111111 &   3, 5, 9, 10, 19\\
    1 &   110111010001111111111101111111 &   3, 7, 9, 10, 11, 23\\
    1 &  110111010011111111111111111111 &   3, 7, 9, 10\\
    1 &   110111110011101111111111111111 &    3, 9, 10, 14\\
    1 &    111011011111111111111011111111 &    4, 7, 22\\
    1 &   111011110011111111010101111111 &   4, 9, 10, 19, 21, 23\\
    1 &    11011110011111111111111111111 &   4, 9, 10\\
    1 &    111011111111101111111111111111 &   4, 14\\
    1 &  111111010011111111010101111111 &     7, 9, 10, 19, 21, 23\\
    1 &   111111011111101111111111111111 &   7, 14\\
    1 &   111111011111111001111111111111 &   7, 16, 17\\
    1 &   111111110010111111011111111111 &    9, 10, 12, 19\\
    1 &  111111110011111011111111111111 &    9, 10, 16\\
    1 &   111111111110111001111111111111 &    12, 16, 17\\
    1 &   111111111110111011111111111111 &    12, 16\\
    1 &   111111111110111111101111111111 &   12, 20\\
    1 &   111111111111111111011111111111 &    19\\
    1 &   111111111111111111111110111111 &   24\\
    1 &   111111111111111111111111011111 &   25\\

\hline
\end{tabular}
\label{tab:source3-part3}
\end{table}

\begin{table}[ht]
\caption{Mutual Constraint Violations for Source 4 in Decreasing Order of Violation Part 1}
\centering
\begin{tabular}{c l l}
\hline\hline
Number & Binary ID for Violation Set & IDs of Violated Constraints \\ [0.5ex] 
\hline

250 &    011111111111111111111111111111 &    1\\
    119 &    111110111111111111111111111111 &     6\\
    67 &   011110111111111111111111111111 &    1, 6 \\
    45 &    111011111111111111111111111111 &    4\\
    23 &    101111111111111111111111111111 &    2\\
    22 &    100101101101111111111111111111 &    2, 3, 5, 8, 11\\
    21 &    000101101101111111111111111111 &   1, 2, 3, 5, 8, 11\\
    19 &    111111111111101111111111111111 &   14\\
    15 &    000111111111111111111111111111 &    1, 2, 3\\
    14 &    001111111111111111111111111111 &   1, 2\\
    13 &    100101111111111111111111111111 &    2, 3, 5\\
    13 &    100111111111111111111111111111 &   2, 3 \\
    12 &   111010111111111111111111111111 &    4, 6\\
    11 &   011111111111101111111111111111 &    1,14\\
    10 &   000101101111111111111111111111 &   1, 2, 3, 5, 8\\
    7 &   000101111111111111111111111111 &   1, 2, 3, 5\\
    7 &   000110111111111111111111111111 &    1, 2, 3, 6\\
    7 &   100101101111111111111111111111 &   2, 3, 5, 8\\
     6 &   100100101101111111111111111111 &    2, 3, 5, 6, 8, 11\\
    5 &   001110111111111111111111111111 &   1, 2, 6\\
    5 &   010111110011111111111111111111 &    1, 3, 9, 10\\
    5 &   101110111111111111111111111111 &    2, 6\\
    5 &   111110111111101111111111111111 &   6, 14\\
     4 &   011011111111111111111111111111 &    1, 4\\
    4 &   011110111111101111111111111111 &    1, 6, 14\\
    4 &   100110111111111111111111111111 &   2, 3, 6\\
    4 &   111111011111111111111111111111 &   7\\
    3 &   000100101101111111111111111111 &    1, 2, 3, 5, 6, 8, 11\\
    3 &   000100101111111111111111111111 &   1, 2, 3, 5, 6, 8\\
    3 &   000101101101101111111111111111 &    1, 2, 3, 5, 8, 11, 14\\
    3 &   100101111101111111111111111111 &    2, 3, 5, 11\\
    3 &   111111110011111111111111111111 &   9, 10\\
    3 &   111111111111111111111011111111 &   22 \\
    2 &   000100111111111111111111111111 &    1, 2, 3, 5, 6\\
    2 &   011010111111111111111111111111 &   1, 4, 6\\
    2 &   011111111111111001111111111111 &   1, 16, 17\\
    2 &   011111111111111111101111111111 &   1, 20\\
    2 &  100100111111111111111111111111 &   2, 3, 5, 6\\
    2 &   100101101101101111111111111111 &   2, 3, 5, 8, 11, 14\\
    2 &  100101111111101111111111111111 &   2, 3, 5, 14\\
    2 &   100111111111101111111111111111 &   2, 3, 14\\
    2 &   101011111111111111111111111111 &    2, 4\\
    2 &   01111111111101111111111111111 &    2, 14\\
    2 &   111111111111111001111111111111 &    16, 17\\
    1 &   000100110011111111111111111111 &   1, 2, 3, 5, 6, 9,10 \\
    1 &   000100111111111111101111111111 &    1, 2, 3, 5, 6, 20\\
    1 &   000110111111101111111111111111 &   1, 2, 3, 6, 14\\
    1 &   000111111111111001111111111111 &   1, 2, 3, 16, 17\\

\hline
\end{tabular}
\label{tab:source4-part1}
\end{table}

\begin{table}[ht]
\caption{Mutual Constraint Violations for Source 4 in Decreasing Order of Violation Part 2}
\centering
\begin{tabular}{c l l}
\hline\hline
Number & Binary ID for Violation Set & IDs of Violated Constraints \\ [0.5ex] 
\hline

 1 &   001101101101111111111011111111 &   1, 2, 5, 8, 11,22 \\
    1 &   001101101111111111111011111111 &   1, 2, 5, 8,  22\\
    1 &   001111111111101111111111111111 &    1, 2, 14\\
    1 &   010110110010111111111111111111 &   1, 3, 6, 9, 10, 12\\
    1 &   010110110011111111111111111111 &    1, 3, 6, 9, 10\\
    1 &   011011110011111111111111111111 &   1, 4, 9, 10\\
    1 &   011110110011111111111111111111 &   1, 6, 9, 10\\
    1 &   011111111111111111111011111111 &    1, 22\\
    1 &   100000111111111111111111111111 &   2, 3, 4, 5, 6\\
    1 &  100001110011111111111111111111 &    2, 3, 4, 5, 9, 10\\
    1 &   100001111111111111111111111111 &    2, 3, 4, 5\\
    1 &    100100101111111111111111111111 &   2, 3, 5, 6, 8\\
    1 &  100100111101111111111111111111 &   2, 3, 5, 6, 11\\
    1 &   100101101111101111111111111111 &   2, 3, 5, 8, 14\\
    1 &   100111110011111111111111111111 &    2, 3, 9, 10\\
    1 &  101101101101111111111011111111 &     2, 5, 8, 11, 22\\
    1 &   101111111111111011111111111111 &    2, 16\\
    1 &   110111110010101111101111111111 &    3, 9, 10, 12, 14, 20\\
    1 &   110111110010111111111111111111 &    3, 9, 10, 12\\
    1 &   110111110011111111111111111111 &   3, 9, 10\\
    1 &   111010110011111111111111111111 &    4, 6, 9, 10\\
    1 &  111011011111111111111110111111 &   4, 7, 24\\
    1 &   111110110011111111111111111111 &   6, 9, 10\\
    1 &  111110111111101001111111111111 &   6, 14, 16, 17\\
    1 &   111111110010111111111111111111 &   9, 10, 12\\
    1 &    111111111111011111111111111111 &   13\\
    1 &   111111111111101011111111111111   & 14, 16\\
    1 &  111111111111111111101111111111 &   20\\

\hline
\end{tabular}
\label{tab:source4-part2}
\end{table}

\begin{table}[ht]
\caption{Mutual Constraint Violations for Source 5 in Decreasing Order of Violation Part 1}
\centering
\begin{tabular}{c l l}
\hline\hline
Number & Binary ID for Violation Set & IDs of Violated Constraints \\ [0.5ex] 
\hline
596  & 111111111101111111111111111111 &  11   0   0   0   0   0   0   0   0\\
    85 &    101111111111111111111111111111 &    '2   0   0   0   0   0   0   0   0\\
    51 &    111111101111111111111111111111 &   8   0   0   0   0   0   0   0   0\\
    44 &   101111111101111111111111111111 &   2  11   0   0   0   0   0   0   0\\
    42 &   111111111111111111111111111110 &    30   0   0   0   0   0   0   0   0\\
    38 &   111101101111010111111111101111 &    5   8  13  15  26   0   0   0   0\\
    34 &    111111101101111111111111111111 &   8  11   0   0   0   0   0   0   0 \\
    34 &   111111111111111111111110111111 &    24   0   0   0   0   0   0   0   0\\
    29 &   111111101111011111111111101111 &   8  13  26   0   0   0   0   0   0\\
    27 &    101111111111111111111111111110 &    2  30   0   0   0   0   0   0   0\\
    26 &    111111101111011111111111111111 &   8  13   0   0   0   0   0   0   0\\
    24 &   111111111111111111111100111111 &   23  24   0   0   0   0   0   0   0\\
    23 &   111110111101111111111111111111 &    6  11   0   0   0   0   0   0   0\\
    22 &   111111111111111110110111111111 &   18  21   0   0   0   0   0   0   0\\
    16 &   111111111101111111111111111110 &    11  30   0   0   0   0   0   0   0\\
    15 &   111111101101011111111111101111 &    8  11  13  26   0   0   0   0   0\\
    14 &    111101101101010111111111101111 &    5   8  11  13  15  26   0   0   0\\
    12 &   101111111101111111111111111110 &   2  11  30   0   0   0   0   0   0\\
    11 &   111111101101011111111111111111 &   8  11  13   0   0   0   0   0   0\\
    10 &    111111111011111111111111111111 &    10   0   0   0   0   0   0   0   0\\
    9 &   111111111111111111111101111111 &   23   0   0   0   0   0   0   0   0\\
    8 &   111110111111111111111111111111 &  6   0   0   0   0   0   0   0   0\\
    8 &   111111101101010111111111101111 &   8  11  13  15  26   0   0   0   0\\
    8 &   111111101111010111111111101111 &   8  13  15  26   0   0   0   0   0\\
    8 &   111111111101111110110111111111 &   11  18  21   0   0   0   0   0   0\\
    6 &   111111111001111111111111111111 &   10  11   0   0   0   0   0   0   0\\
    6 &   111111111101011110110111111111 &   11  13  18  21   0   0   0   0   0\\
    5 &   111111111111111011111111011111 &    16  25   0   0   0   0   0   0   0\\
    4 &   101111111111111111111110111111 &    2  24   0   0   0   0   0   0   0\\
    4 &  111111111111011110110111111111 &    13  18  21   0   0   0   0   0   0\\
    3 &  101111101111111111111111111111 &    2   8   0   0   0   0   0   0   0\\
    3 &   '111111111101111111111101111111 &   11  23   0   0   0   0   0   0   0\\
    3 &  111111111111111111111111110111 &   27   0   0   0   0   0   0   0   0\\
    2 &   101111101111011111111111111111 &   2   8  13   0   0   0   0   0   0\\
    2 &   101111111111111110110111111111 &    2  18  21   0   0   0   0   0   0\\
     2 &    111101101101010110110111101111 &    5   8  11  13  15  18  21  26   0\\
    2 &   111110111111111111111111111110 &    6  30   0   0   0   0   0   0   0\\
    2 &  111111101111011111111100101111 &   8  13  23  24  26   0   0   0   0\\
    2 &   111111111101111111111111011111 &    11  25   0   0   0   0   0   0   0\\
    2 &   111111111111111111111111011111 &   25   0   0   0   0   0   0   0   0\\
    1 &   011111110111011010010011011111 &    1   9  13  16  18  19  21  22  25\\
    1 &   011111110111111010010011011111 &   1   9  16  18  19  21  22  25   0\\
    1 &   011111111101110011111111011111 &   1  11  15  16  25   0   0   0   0\\
    1 &   101101101101010111111111101111 &     2   5   8  11  13  15  26   0   0\\
    1 &   101101101111010111111111101110 &   2   5   8  13  15  26  30   0   0\\
    1 &   101101101111010111111111101111 &    2   5   8  13  15  26   0   0   0\\
    1 &   101110111101111111111111111111 &   2   6  11   0   0   0   0   0   0\\
    1 &   101111101101011111111111101111 &   2   8  11  13  26   0   0   0   0\\

\hline
\end{tabular}
\label{tab:source5-part1}
\end{table}

\begin{table}[ht]
\caption{Mutual Constraint Violations for Source 5 in Decreasing Order of Violation Part 2}
\centering
\begin{tabular}{c l l}
\hline\hline
Number & Binary ID for Violation Set & IDs of Violated Constraints \\ [0.5ex] 
\hline
 1 &   101111101101011111111111111110  &  2   8  11  13  30   0   0   0   0\\
   1 &  101111101101111111111110111111 &   2   8  11  24   0   0   0   0   0\\
    1 &   101111101101111111111111111110 &    2   8  11  30   0   0   0   0   0\\
    1 &   101111101111011111111111101111 &   2   8  13  26   0   0   0   0   0 \\
    1 &   101111101111011111111111111110 &    2   8  13  30   0   0   0   0   0\\
    1 &   101111110101011110110111111111 &   2   9  11  13  18  21   0   0   0\\
    1 &   101111111101011110110111111111 &   2  11  13  18  21   0   0   0   0\\
    1 &   101111111101111110110111111111 &   2  11  18  21   0   0   0   0   0\\
    1 &   101111111101111111111111011110 &    2  11  25  30   0   0   0   0   0\\
    1 &  101111111101111111111111011111 &   2  11  25   0   0   0   0   0   0\\
    1 &   101111111111011110110111111111 &    2  13  18  21   0   0   0   0   0\\
    1 &   111010111101111111111111111111 &   4   6  11   0   0   0   0   0   0\\
    1 &   111011101101111111111111111111 &   4   8  11   0   0   0   0   0   0\\
     1 &  111011111101111111111111111111 &   4  11   0   0   0   0   0   0   0\\
    1 & 111101101001010111111111101111 &     5   8  10  11  13  15  26   0   0\\
    1 &   111101101111011111111111101111 &   5   8  13  26   0   0   0   0   0\\
    1 &   111110101101010111111111101110 &    6   8  11  13  15  26  30   0   0\\
    1 &    111110101101011111111111101111 &   6   8  11  13  26   0   0   0   0\\
    1 &   111110101101111111111111011111 &   6   8  11  25   0   0   0   0   0\\
    1 &   111110101101111111111111111111 &    6   8  11   0   0   0   0   0   0\\
     1 &   111110101111011111111111111111 &   6   8  13   0   0   0   0   0   0\\
    1 &  111111100111011110110111101111 &    8   9  13  18  21  26   0   0   0\\
    1 &  111111100111011110110111111111 &   8   9  13  18  21   0   0   0   0\\
    1 &  111111101001011111111111111111 &   8  10  11  13   0   0   0   0   0\\
    1 &   111111101101011110110111111111 &   8  11  13  18  21   0   0   0   0\\
    1 &  111111101111010111111111101110 &   8  13  15  26  30   0   0   0   0\\
    1 &   111111101111011011111011011111 &    8  13  16  22  25   0   0   0   0\\
    1 &   111111101111011011111111011111 &   8  13  16  25   0   0   0   0   0\\
    1 &   111111101111011110110111111111 &    8  13  18  21   0   0   0   0   0\\
    1 &  111111101111011111111101111111 &  8  13  23   0   0   0   0   0   0\\
    1 &   111111101111011111111110101111 &   8  13  24  26   0   0   0   0   0\\
    1 &   111111101111011111111111111110 &   8  13  30   0   0   0   0   0   0\\
    1 &   111111101111110111111111101110 &   8  15  26  30   0   0   0   0   0\\
    1 &   111111101111111111111100111111 &    8  23  24   0   0   0   0   0   0\\
    1 &  111111101111111111111110111111 &   8  24   0   0   0   0   0   0   0\\
    1 &  111111101111111111111111111110 &    8  30   0   0   0   0   0   0   0\\
    1 &  111111110111011110110111111111 &   9  13  18  21   0   0   0   0   0\\
    1 &    111111111011111111111111111110 &   10  30   0   0   0   0   0   0   0\\
    1 &   111111111101011110110111110111 &   11  13  18  21  27   0   0   0   0\\
    1 &   111111111101111111111100111111 &   11  23  24   0   0   0   0   0   0\\
    1 &   111111111101111111111111111011 &    11  28   0   0   0   0   0   0   0\\
    1 &   111111111111111011111011011110 &   16  22  25  30   0   0   0   0   0\\
    1 &   111111111111111110110100111111 &   18  21  23  24   0   0   0   0   0\\
    1 &   111111111111111110110110111111 &   18  21  24   0   0   0   0   0   0\\
    1 &   111111111111111111111111011110 &   25  30   0   0   0   0   0   0   0\\
\hline
\end{tabular}
\label{tab:source5-part2}
\end{table}

\section{Example: Simplicial Complexes for Source 2 } 

The simplicial complex for each data source is generated by listing the mutual constraint violation sets for the data items in each data source. These lists are shown for Sources 2, 3, 4 and 5 in the right hand columns of the  long tables in Appendix A. 
The lexicographically sorted list for Source 2, inferrred from the right hand columns of the  table for Source 2  is shown in Table~\ref{tab:source2-simpcomplex}. This was the input to chomp-simplicial for computation of the betti numbers for Source 2.

\begin{table}[ht]
\caption{Simplicial Complex Generators for Source 2: Maximal Mutual Constraint Violation Sets}
\centering
\begin{tabular}{r r r r r r r r}

1 \\
 1 & 2\\
 1 & 2 & 3\\
 1& 2 & 3 & 5\\
 1 & 2 & 3 & 5 & 8\\
 1 & 2 & 3 & 5 & 8 & 11\\
 1 & 2 & 3 & 5 & 8 & 11 & 16\\
 1 & 2 & 3  & 9 & 10 & 12\\
 1 & 2 &  3 & 14\\
 1 & 2 & 14\\
 1 & 3 & 9  &10\\
 1 & 4\\
 1 & 4 &14\\
 1 & 6\\
 1 & 9& 10\\
 1 & 14\\
 1 & 16 &17 & 18\\
 1 &18\\
2\\
 2 & 3\\
 2 & 3 & 4\\
 2 & 3 & 4 & 5\\
 2 & 3 & 4 & 5  & 8 & 11\\
 2 & 3 & 5\\
 2 & 3 & 5 & 8\\
 2 & 3  & 5 & 8 & 11\\
 2  & 3 & 5 & 9 & 10\\
 2  & 3 & 16\\
 2 & 4\\
 2 & 5\\
 2 & 5 & 8 & 13 & 15\\
 2 & 16 & 17\\
 3 & 4 & 9 & 10\\
 3 & 9 & 10\\
 3  & 9 & 10 & 12\\
 4 \\
 6 \\
 7 \\
 9 & 10\\
 9 & 10 & 12\\
 9 & 10 & 12 & 16 & 17 & 18 & 22 & 23\\
  13\\
 14\\
 15\\
 16 & 17\\
 16 & 17 &18\\
 20\\
 24\\

\end{tabular}
\label{tab:source2-simpcomplex}
\end{table}

\end{document}